\documentclass[a4paper,10pt,reqno]{amsart}
\usepackage[utf8]{inputenc}
\usepackage{amsmath,amsfonts,amssymb,mathrsfs,latexsym,bbm,enumitem,xcolor,hyperref}
\usepackage[all,cmtip]{xy}
\newtheorem{definition}{Definition}
\newtheorem{theorem}[definition]{Theorem}
\newtheorem{lemma}[definition]{Lemma}
\newtheorem{coro}[definition]{Corollary}
\newtheorem{proposition}[definition]{Proposition}
\newtheorem{remark}[definition]{Remark}
\begin{document}
\title{Witt Group of Nondyadic Curves}
\author{Nanjun Yang}
\thanks{The author is partially supported by the National Natural Science Foundation of China (Grant No. 12201336). He would like to thank Piotr Achinger, Lei Fu, Joe Harris, Chuangqiang Hu, Bruno Kahn, Will Savin and Weizhe Zheng for answering my questions. He specially thanks the helpful suggestions given by referees.}
\address{Nanjun Yang\\Beijing Institute of Mathematical Sciences and Applications\\Huairou District\\Beijing China}
\email{ynj.t.g@126.com}
\begin{abstract}
Witt group of real algebraic curves has been studied since Knebusch in 1970s. But few results are known if the base field is non-Archimedean except the hyperelliptic case by works of Parimala, Arason et al.. In this paper, we compute the derived Witt groups of smooth proper curves over nondyadic local fields with \(char\neq2\) by reduction, with a general study of the existence of Theta characteristics.
\end{abstract}
\maketitle
\tableofcontents
\section{Introduction}\label{introduction}
Denote by \(Sm/k\) the category of smooth separated schemes over a field \(k\), by \(H^i(-)=H^i_{\textrm{\'et}}(-,\mathbb{Z}/2)\) the mod 2 \'etale cohomologies and by \(\mathcal{H}^i\) its associated Zariski sheaf. A local field \(K\) is either a finite extension of \(\mathbb{Q}_p\) or being equal to \(\mathbb{F}_{p^n}((t))\). We say that \(K\) is nondyadic if \(p>2\).

The Witt group \(W(X)\) of a scheme \(X\) (\cite{Kn1}, \(2\in O_X^{\times}\)) is to consider the semiring (\(\oplus,\otimes\)) of the set of vector bundles \(E\) with nondegenerate symmetric inner product, modulo by (the ideal of) those with Lagrangians, namely a subbundle \(L\subseteq E\) such that \(L=L^{\perp}\). For a general field \(k (char(k)\neq2)\), the \(W(k)\) classifies anisotropic quadratic forms over \(k\). We have the following well-known facts (\cite{Lam})
\[W(k)=\begin{cases}\mathbb{Z}&k=\mathbb{R}\\\mathbb{Z}/2^{\oplus4}&k\textrm{ nondyadic local },\sqrt{-1}\in k\\\mathbb{Z}/4^{\oplus2}&k\textrm{ nondyadic local },\sqrt{-1}\notin k\end{cases}.\]
The ideal \(I(k)\) of even dimensional inner product spaces is called the fundamental ideal. The Milnor conjecture (\cite{V1},\cite{OVV}) says that we have an equality
\[I^n(k)/I^{n+1}(k)=K_n^M(k)/2=H^n(k)\]
where the \(K_*^M(k)\) is the Milnor K-theory of \(k\) (\cite{Mr}).

Later Balmer defined in \cite{Ba1} the Witt group \(W(K,\#,\bar{\omega})\) for any triangulated category \((K,\#,\bar{\omega})\) with \(\delta\)-duality (\(\delta=\pm1\)). Objects \(P\in K\) with symmetric isomorphism \(\varphi:P\cong P^{\#}\) are called symmetric objects, denoted by the pair \((P,\varphi)\). If there is an \(A\in K\) such that \(P\) is a mapping cone of some \(u:A\to A^{\#}[-1]\) being compatible with \(\bar{\omega}\), the \((P,\varphi)\) is called neutral or Lagrangian. Then define \(W(K,\#,\bar{\omega})\) to be the quotient of monoid of symmetric objects by neutral objects.

If we consider symmetric morphisms \(\varphi:P\to P^{\#}\) not being an isomorphism, the mapping cone of \(\varphi\) is a symmetric object in the sense of shifted \((-\delta)\)-duality structure \((K,\#,\omega)[1]=(K,[1]\circ\#,(-\delta)\bar{\omega})\). Then we can defined the derived Witt group
\[W^i(K,\#,\bar{\omega})=W((K,\#,\bar{\omega})[i])\]
which satisfies \(W^i(K,\#,\bar{\omega})=W^{i+4}(K,\#,\bar{\omega})\). For any regular scheme \(X\), we define \(W^i(X)=W^i(D^b(Vect(X)),*,\bar{\omega})\), taking \(K\) to be the bounded derived category of the category of vector bundles on \(X\). If \(i\equiv0 (\textrm{mod }4)\), the \(W^i(X)\) is equal to \(W(X)\) defined in \cite{Kn1} by the main theorem of \cite{Ba2}. 

An important tool to compute \(W^*(X)\) is the Gersten-Witt complex (\cite{BW})
\[W_X:0\to W(K(X))\to\oplus_{x\in X^{(1)}}W(k(x))\to\cdots\to\oplus_{x\in X^{(dim(X))}}W(k(x))\to0\]
defined in \cite{BW} by a coniveau filtration of \(D^b(Vect(X))\). It is known (\cite[Theorem 10.1]{BW}) that if \(dim(X)\leq 3\), we have \(H^i(W_X)=W^i(X),0\leq i\leq 3\)

On the other hand, given \(X\in Sm/k (char(k)\neq 2)\), we can define a Rost-Schmid complex (\cite{Mo})
\[C^n(X,\textbf{I}^m)=\oplus_{x\in X^{(n)}}I^{m-n}(k(x),(\omega_X|_x)^{\vee})\]
and cohomologies
\[H^n(X,\textbf{I}^m)=H^n(C^*(X,\textbf{I}^m)),\]
where differentials are given by cohomological transfers between twisted Witt groups. These cohomologies are the Nisnevich (or Zariski) cohomologies of the Nisnevich sheaf \(\textbf{I}^m\), by taking unramified elements in \(I^m(K(X))\). The same definition applies if we replace \(\textbf{I}^*\) by \(\textbf{K}_*^M/2\). We have an isomorphism \(C^*(X,\textbf{W})=W_X\) by \cite[Lemma 8.4]{BW} and \cite[Remarque 7.3.1]{F1}. Consequently we have the key equality:
\begin{equation}\label{key}W^i(X)=H^i(X,\textbf{W}),0\leq i\leq3\end{equation}
if \(dim(X)\leq3\).

For smooth projective curves \(X\) over an algebraically closed field \(k\), its (derived) Witt groups were computed in \cite{Zib}. When \(k\) is real closed, there has been intensive study on its Witt group (\cite{Kn},\cite{Mon}). Suppose \(X\) has genus \(g\) and \(n\) real closed connected components. We have the result in \cite{Mon}:
\[W(X)=\begin{cases}\mathbb{Z}^{\oplus n}\oplus\mathbb{Z}/2^{\oplus g}&X(k)\neq\emptyset\\\mathbb{Z}/4\oplus\mathbb{Z}/2^{\oplus g}&X(k)=\emptyset,X(k\sqrt{-1})\textrm{ connected}\\\mathbb{Z}/2^{\oplus 2g+1}&X(k)=\emptyset,X(k\sqrt{-1})\textrm{ not connected}\end{cases}.\]
When \(k\) is a local field, the \cite{AEJ} computed the cardinality of \(W(X)\) and the \cite{PS} aimed at the case of hyperelliptic curves with good reduction. The \cite{PSr} computed the filtration \(H^0(X,\textbf{I}^n)/H^0(X,\textbf{I}^{n+1})=\mathcal{H}^n(X)\) if the Theta charcteristics \(\sqrt{\omega_X}\) exists. None of these work computed the 4-torsions in \(W^*(X)\). The \cite{Sha} and \cite{PSc} discussed the existence of Theta characteristics when the curve is hyperelliptic or has degenerate semistable reduction (i.e. union of \(\mathbb{P}^1\)). The \cite{R} gave a motivic explaination of the existence of Theta characteristics, as well as its connection to the Steenrod square.
\section{Statement of the Results}
Suppose that \(X_K\) is a connected smooth proper curve over a nondyadic local field \(K\) with \(char(K)\neq2\) and residue field \(k\). One finds a model \(X\) which is regular and proper flat over \(O_K\) with generic fiber \(X_K\).  In this paper, we compute general \(W^*(X_K)\) by studying the reduction \(X_k=X\times_{O_K}k\).

The algorithm runs as the following: Denote by \(G=Ker(Pic^0_{X_k/k}\to Pic^0_{\widetilde{X_k^{red}}/k})\) and \(p:\widetilde{X_k^{red}}\to X_k\) the normalization map. We have an exact sequence
\[O^{\times}(\widetilde{X^{red}_k})/2\to\frac{\oplus_{x'\in p^{-1}((X_k)_{sing})}H^1(x')}{\oplus_{x\in (X_k)_{sing}}H^1(x)}\to H^1(k,{_2}G)\to0.\]
by Proposition \ref{picard} (\({_2}G\) is the \(2\)-torsion part of \(G\)). Define
\[G(X_k)=Im(\oplus_{x'\in p^{-1}((X_k)_{sing}),\sqrt{-1}\notin k(x')}H^1(x')\to H^1(k,{_2}G)\to H^1(k,{_2}Pic(X_k)))\]
(Definition \ref{G1}), where \((X_k)_{sing}\) are the singularities of \(X_k^{red}\). Furthermore we generalize the Merkurjev's pairing (Definition \ref{Merkurjev}) to singular \(X_k\) as the following
\[\begin{array}{ccccc}CH_0(X_k)/2&\times&H^1(X_k)&\to&H^1(k)=\mathbb{Z}/2\\x&,&f&\mapsto&Tr_{k(x)/k}(f|_{k(x)})\end{array}.\]
Define by that pairing
\[S(X_k)=Coker(CH_0(X_k)/2\xrightarrow{m}H^1(X_k)^{\vee}).\]
By Proposition \ref{unramified}, we have \(S(X_k)=\mathcal{H}^3(X_K)\) and the \(m\) is injective if \(X_k\) is reduced. 

Finally we define
\[tr(X_k)=\#\{Y\textrm{ irr. comp. of }\widetilde{X_k^{red}}|[O(Y):k]\textrm{ is odd}\}.\]
Then we have the followings in the context above (Theorem \ref{h1} and Theorem \ref{h0}):
\begin{theorem}
Suppose that \(K\) is a nondyadic local field with \(char(K)\neq 2\) and that \(X_K\) has a rational point and is connected. The \(4\)-torsion group \(W^i(X_K)\) (\(=0\) for \(i=2,3\)) satisfies (\(l\) is the length function)
\[l(W(X_K))=dim(S(X_k))+2dim(H^2(X_k))+q+1,\]
\[dim(2W(X_K))=\begin{cases}dim(G(X_k))+tr(X_k)+\delta+1&\sqrt{-1}\notin K\\0&\sqrt{-1}\in K\end{cases}.\]
\[l(W^1(X_K))=2dim(H^2(X_k))+q+1-dim(S(X_k))\]
\[dim(2W^1(X_K))=\begin{cases}q+tr(X_k)&\sqrt{-1}\notin K\\0&\sqrt{-1}\in K\end{cases}.\]
\end{theorem}
Here the \(q=1\) (resp. \(q'=1\)) if the Theta characteristics \(\sqrt{\omega_{X_K}}\) (resp. \(\sqrt{\omega_{X_{K\sqrt{-1}}}}\)) exists, otherwise \(q=0\) (resp. \(q'=0\)). On the other hand, define
\[\delta=\begin{cases}0&q'=0\\0&q'=1, \omega_{X_k}\in G(X_k)\\1&\textrm{else}\end{cases}.\]

The main difficulty of the computation is the identification of the \(4\)-torsion part of \(W(X_K)\) and that of \(q,\delta\). By \eqref{key} it suffices to compute \(H^*(X_K,\textbf{W})\). We have the motivic Bockstein spectral sequence (\S\ref{bockstein})
\[E_1^{p,q}=H^{p,q}(X_K,\mathbb{Z}/2)\oplus H^{p+2,q}(X_K,\mathbb{Z}/2),\]
strongly converging to \(H^{p-q}(X_K,\textbf{W})\). Its differentials on \(E_1\)-page (resp. \(E_n\)-page, \(n\geq2\)) are linear combinations of \(Sq^2\) and \(Sq^1\) (resp. higher Bocksteins). For our \(X_K\), higher Bocksteins vanish. By analyzing the Bockstein SS, we can identify the \(4\)-torsion part.

For \(q,\delta\) we aim at concrete computation (the \(q'\) is easier). We note that one can use Hurwitz's formula and some map \(X_K\to\mathbb{P}^1\) to compute the multiplicites of \(\omega_{X_K}\) at each point as a Weil divisor. If all multiplicities are even, the \(\omega_{X_K}\) is a square. But the converse is not true.

The existence of \(\sqrt{\omega_{X_K}}\) is reduced to that of \(\sqrt{\omega_{X_k}}\) and is equivalent to the vanishing of \(Sq^2\) on \(H^{2,2}_M\) (Proposition \ref{sq}). The \(\omega_{X_{\bar{k}}}\) is a square if and only if the self-intersections of components of \(X_{\bar{k}}\) are even (Proposition \ref{ac}). Denote by
\[R=Im(H^0(k,{_2}Pic(\widetilde{X_k^{red}}))\to H^1(k,{_2}G))\]
which can be computed by the Tate pairing (Proposition \ref{semisimple}). So we have
\[dim(H^2(X_k))=dim(Pic(\widetilde{X_k^{red}})/2)+dim(H^1(k,{_2}G)/R).\]
The \(p^*\omega_{X_k}\) is a square if and only if \(\omega_{X_{\bar{k}}}\) is a square and the theta-form (\cite[1.12]{Harr}) vanishes on \(R\) (Proposition \ref{sq1}). In this case we write
\[p^*\omega_{X_k}=2M+div(f)\]
as Cartier divisors with supports contained in \((X_k)_{reg}\). The \((f(x))_{p(x)\in(X_k)_{sing}}\) defines a class \(\Lambda(\omega_{X_k})\) in \(H^1(k,{_2G})/R\) (Definition \ref{lambda}), which is zero if and only if \(\omega_{X_K}\) is a square.

The Corollary \ref{algorithm} discusses the computation of \(\Lambda(\omega_{X_k})\) when each component of \(\widetilde{X_k^{red}}\) is either a hyperelliptic curve or \(\mathbb{P}^1\).

A complete computation of \(W(X_K)\) when \(X_K\) is an elliptic curve is given in \S\ref{elliptic}.

For convenience, we give a list of frequently used notations in this paper:
\[\begin{array}{c|c}
Sm/k						&\textrm{The category of smooth and separated schemes over \(k\)}\\
H^*(X)					&\textrm{\'Etale cohomology with mod 2 coefficients of \(X\)}\\
\chi(X,2)				&\textrm{The Euler Poincar\'e characteristic of }H^*(X)\\
H^*(x)					&\textrm{\'Etale cohomology with mod 2 coefficients of the residue field of }x\\
H^{*,*}_{M}			&\textrm{Motivic cohomology with mod 2 coefficients}\\
H^{*,*}_{W}			&\textrm{Witt motivic cohomology}\\
\mathcal{H}^*(X)	&H^0(X,H^*(-)_{Zar})\\
X\sqrt{-1}				&X\times_kk\sqrt{-1}\\
X_{sing}				&\textrm{Singular points of \(X^{red}\)}\\
X_K						&\textrm{The Generic fiber}\\
X_k						&\textrm{The Special fiber}\\
\tau						&\textrm{The motivic Bott element}\\
\eta						&\textrm{The Hopf element}\\
\rho						&[-1]\in H^1(k)\\
A/\rho,A/\eta			&\textrm{Mapping cone of }A\wedge\rho,A\wedge\eta\\
_aM						&\textrm{The \(a\)-torsions of }M\\
\sigma					&\textrm{The Frobenius action on }k
\end{array}\]
\section{The Motivic Stable Homotopy Category}\label{SH}
The main references of this section are \cite{Mo1}, \cite{B} and \cite{DLORV}. Let \(sShv_{\bullet}(Sm/k)\) be the category of pointed simplicial sheaves over \(Sm/k\) for the Nisnevich topology. We localize it by the morphisms
\[F\wedge\mathbb{A}^1_+\longrightarrow F\]
for every \(F\in sShv_{\bullet}(Sm/k)\), obtaining the homotopy category \(\mathcal{H}_{\bullet}(k)\).

Define \(\mathcal{SH}(k)\) to be the homotopy category of the \(\mathbb{P}^1\) or \(\mathbb{G}_m-S^1\) spectra of \(sShv_{\bullet}(Sm/k)\) (see \cite[2.3]{DLORV}). It is triangulated and symmetric monoidal with respect to the smash product \(\wedge\). There is an adjunction
\[\Sigma^{\infty}:\mathcal{H}_{\bullet}(k)\rightleftharpoons\mathcal{SH}(k):\Omega^{\infty}.\]
In \(\mathcal{SH}(k)\) we have spheres \(S^{n,0}=S^n\) and \(S^{n,n}=\mathbb{G}_m^{\wedge n}\). They are invertible objects with resepct to the smash product. We define \(S^{p,q},p,q\in\mathbb{Z}\) by the rule \(S^{p,q}\wedge S^{r,s}=S^{p+r,q+s}\).

Now we look at the morphism
\[\begin{array}{ccc}\mathbb{A}^2\setminus 0&\longrightarrow&\mathbb{P}^1\\(x,y)&\longmapsto&[x:y]\end{array}.\]
It induces a morphism
\[\mathbb{G}_m\wedge\mathbb{G}_m\wedge S^1\longrightarrow\mathbb{G}_m\wedge S^1,\]
which is the suspension of a (unique) morphism \(\eta\in Hom_{\mathcal{SH}(k)}(S^{1,1},S^{0,0})\), called the Hopf map.

For any \(E\in\mathcal{SH}(k)\) and \(n,m\in\mathbb{Z}\), define \(\pi_n(E)_m\) to be the (Nisnevich) sheafification of the presheaf on \(Sm/k\)
\[X\longmapsto[\Sigma^{\infty}X_+\wedge S^n,E(m)[m]]_{\mathcal{SH}(k)}.\]
Then a morphism \(E_1\longrightarrow E_2\) between spectra is a weak equivalence if and only if it induces an isomorphism
\[\pi_n(E_1)_m\cong\pi_n(E_2)_m\]
for every \(n,m\in\mathbb{Z}\).

Define
\[\mathcal{SH}_{\leq -1}=\{E\in\mathcal{SH}(k)|\pi_n(E)_m=0, \forall n\geq 0,m\in\mathbb{Z}\},\]
\[\mathcal{SH}_{\geq 0}=\{E\in\mathcal{SH}(k)|\pi_n(E)_m=0, \forall n<0,m\in\mathbb{Z}\}.\]
They give a \(t\)-structure (see \cite[5.2]{Mo1}) on \(\mathcal{SH}(k)\) where \(\mathcal{SH}_{\geq 0}\) can be described as the smallest full subcategory of \(\mathcal{SH}(k)\) being stable under suspension, extensions and direct sums, containing \(\Sigma^{\infty}X_+\wedge\mathbb{G}^{\wedge i}_m\) for every \(X\in Sm/k\) and \(i\in\mathbb{Z}\) (see \cite[Proposition 2.1.70]{Ay}). Its heart \(\mathcal{SH}^{\heartsuit}\) is equivalent to the category of homotopy modules (see \cite[Definition 5.2.4]{Mo1}), where the equivalence is given by
\[E\longmapsto\pi_0(E)_*.\]
Examples of homotopy modules are the Nisnevich sheaves \(\textbf{I}^*, \textbf{K}^M_*/2\) defined by \(0^{th}\) cohomology of Rost-Schmid complexes (\cite{Mo}, \S\ref{introduction}), where sections are defined by exact sequences (\(X\in Sm/k\))
\[0\to\textbf{I}^*(X)\to I^n(K(X))\to\oplus_{x\in X^{(1)}}I^n(k(x),\omega_x^{\vee})\]
\[0\to\textbf{K}^M_*/2(X)\to K^M_*/2(K(X))\to\oplus_{x\in X^{(1)}}K^M_*/2(k(x)).\]
Here the twist \(\omega_x^{\vee}\) makes the differential maps independent of the choices of the uniformizer of \(O_{X,x}\).

Define \(\mathcal{SH}^{eff}(k)\) to be smallest triangulated full subcategory on \(\mathcal{SH}(k)\) containing \(\Sigma^{\infty}X_+\) for every \(X\in Sm/k\). The functor from \(\mathcal{SH}^{eff}(k)\)
\[E\longmapsto\pi_*(E)_0\]
is conservative (see \cite[Proposition 4]{B}). Define
\[\mathcal{SH}^{eff}_{\leq_e-1}=\{E\in\mathcal{SH}^{eff}(k)|\pi_n(E)_0=0, \forall n\geq 0\},\]
\[\mathcal{SH}^{eff}_{\geq_e0}=\{E\in\mathcal{SH}^{eff}(k)|\pi_n(E)_0=0, \forall n<0\}.\]
They give a \(t\)-structure (see \cite[\S 3]{B}) of \(\mathcal{SH}^{eff}(k)\) where \(\mathcal{SH}^{eff}_{\geq_e0}\) is the smallest full subcategory of \(\mathcal{SH}^{eff}(k)\) being stable under suspension, extensions, direct sums and containing \(\Sigma^{\infty}X_+\) for every \(X\in Sm/k\). The functor \(\mathcal{SH}^{eff,\heartsuit}\longrightarrow Ab(Sm/k)\) sending \(E\) to \(\pi_0(E)_0\) is conservative (see \cite[Proposition 5]{B}). We further define
\[\mathcal{SH}^{eff}(k)(n)=\mathcal{SH}^{eff}(k)\wedge\mathbb{G}_m^{\wedge n}\]
for any \(n\in\mathbb{Z}\), which has a \(t\)-structure obtained by shifting that of \(\mathcal{SH}^{eff}(k)\) by \(\mathbb{G}_m^{\wedge n}\).

There is an adjunction
\[i_n:\mathcal{SH}^{eff}(k)(n)\rightleftharpoons\mathcal{SH}(k):r_n\]
by \cite[Theorem 4.1]{N} where \(i_n\) is the inclusion. The functor \(r_n\) is \(t\)-exact and \(i_n\) is right \(t\)-exact. Moreover, we have a functor
\[i^{\heartsuit}:\mathcal{SH}(k)^{eff,\heartsuit}\longrightarrow\mathcal{SH}(k)^{\heartsuit},\]
whose essential image consists of effective homotopy modules. Define
\[f_n=i_n\circ r_n.\]
We have natural isomorphisms \(r_n\circ i_n\cong Id\) and \(f_{n+1}\circ f_n\cong f_{n+1}\), where the latter induces a natural transformation \(f_{n+1}\longrightarrow f_n\) (see \cite[\S 4]{DLORV}). We define \(s_n(E)\) for any \(E\in\mathcal{SH}(k)\) by the functorial distinguished triangle
\[f_{n+1}(E)\longrightarrow f_n(E)\longrightarrow s_n(E)\longrightarrow f_{n+1}(E)[1].\]

Then we define the spectra of motivic cohomologies (\cite[pp. 1134]{B}):
\begin{definition}
\[H_W\mathbb{Z}=f_0(\textbf{I}^*), H_{\mu}\mathbb{Z}/2=f_0(\textbf{K}_*^M/2).\]
\end{definition}
Applying \(Hom_{\mathcal{SH}(k)}(\Sigma^{\infty}X_+,-\wedge S^{*,*})\) to \(H_{\mu}\mathbb{Z}/2, H_W\mathbb{Z}\) gives \(H^{*,*}_M, H^{*,*}_W\), respectively. These two spectra can be related by the following distinguished triangle (\cite[Lemma 20]{B}):
\begin{proposition}\label{eta}
\[H_W\mathbb{Z}\wedge\mathbb{G}_m\xrightarrow{\eta}H_W\mathbb{Z}\to H_{\mu}\mathbb{Z}/2\oplus H_{\mu}\mathbb{Z}/2[2]\to H_W\mathbb{Z}\wedge\mathbb{P}^1.\]
\end{proposition}
\section{Motivic Cohomology and Bockstein Spectral Sequence}\label{bockstein}
The norm residue theorem in \cite[Theorem 1.1]{GL} states that for any \(X\in Sm/k, char(k)\neq 2\), the motivic cohomologies \(H^{p,q}_M(X,\mathbb{Z}/2)\), (abbr. \(H^{p,q}_M(X)\), see \cite[Definition 3.4]{MVW} for a definition of motivic cohomologies \(H^{p,q}_M(X,A)\) with coefficients in \(A\)) is equal to
\[\mathbb{H}^p_{Zar}(X,\tau_{\leq q}R\epsilon_*(\mathbb{Z}/2)),\]
where \(\epsilon:X_{\textrm{\'et}}\to X_{Zar}\) is the natural functor and \(\tau\) is the truncation. This implies that \(H^{p,q}_M=H^p\) if \(p\leq q\). Over a field \(k, char\neq 2\), there is a ring isomorphism (\cite{V1})
\[H^{*,*}_M(k)=H^*(k)[\tau]\]
where \(H^n(k)=H^{n,n}_M(k)\) and \(H^{0,1}_M(k)=\mathbb{Z}/2\cdot\tau\), spanned by the motivic Bott element \(\tau\). Moreover, we have the following facts (\cite[Lecture 3, Lecture 19]{MVW})
\[H^{p,q}_M(X)=\begin{cases}CH^q(X)/2&p=2q\\0&p>q+dim(X)\textrm{ or }p>2q\end{cases}\]

For \(X\in Sm/k\), we have Steenrod operations \(Sq^i\) defined on \(H^{*,*}_M(X)\). Here it suffices to know when \(i=1,2\). The \(Sq^1\) is the Bockstein map induced by the exact sequence
\[0\to\mathbb{Z}\xrightarrow{2}\mathbb{Z}\to\mathbb{Z}/2\to0,\]
namely the composite
\[H^{*,*}_M(X)\to H^{*+1,*}_M(X,\mathbb{Z})\to H^{*+1,*}_M(X).\]
The \(Sq^1\) satisfies the Leibnitz rule and \(\rho:=Sq^1(\tau)=-1\in k^*/2=H^{1,1}_M(k)\). The \(Sq^2:H^{*,*}_M\to H^{*+2,*+1}_M\) can be described by the following properties (\cite{V}):
\[Sq^2(xy)=xSq^2(y)+Sq^2(x)y+\tau Sq^1(x)Sq^1(y)\]
\[Sq^2(x)=\begin{cases}x^2&x\in H^{2,1}_M=Pic/2\\0&x\in H^{1,1}_M=H^1\end{cases}.\] 

In \cite[2.2]{FY}, we established the Bockstein spectral sequence (abbr. BSS) of the exact couple for general \(X\in Sm/k\)
\[
	\xymatrix
	{
		\oplus_{p,q}H^{p,q}_W(X)\ar[rr]^{\eta}	&																				&\oplus_{p,q}H^{p,q}_W(X)\ar[ld]_{\pi}\\
																						&\oplus_{p,q}E(W)_1^{p,q}(X)\ar[lu]^{\partial}	&
	},
\]
with
\[E(W)_1^{p,q}(X)=H^{p,q}_{M}(X)\oplus H^{p+2,q}_{M}(X)\]
and differential maps on \(E_1\) explained in \cite[Lemma 2.6]{FY}:
\[\begin{array}{ccc}H^{p,q}_M\oplus H^{p+2,q}_M&\to&H^{p+2,q+1}_M\oplus H^{p+4,q+1}_M\\(x,y)&\mapsto&(Sq^2(x)+\tau y,Sq^1Sq^2Sq^1(x)+Sq^2(y)+\rho Sq^1(y))\end{array},\]
whereas for \(E_2\) and higher pages, the differentials are higher Bocksteins. Here the \(\eta\) is the Hopf element (\cite[Definition 4.1]{Y1}) with degree \((-1,-1)\) and
\[H_W^{*,*}(X)=Hom_{\mathcal{SH}(k)}(\Sigma^{\infty}X_+,H_W\mathbb{Z}\wedge S^{*,*})\]
is the Witt motivic cohomology defined by Backmann (\cite[pp. 1134]{B}), where \(\mathcal{SH}(k)\) is the motivic stable homotopy category over \(k\) (\cite{Mo1}) and \(S^{p,q}=\mathbb{G}_m^{\wedge q}\wedge S^{p-q}\). The \(E_1\)-page is precisely \(Hom_{\mathcal{SH}(k)}(\Sigma^{\infty}X_+,H_W\mathbb{Z}/\eta\wedge S^{*,*})\) and differentials are given by the boundary maps \(H_W\mathbb{Z}/\eta\to H_W\mathbb{Z}/\eta\wedge\mathbb{P}^1\) and the identification of \(H_W\mathbb{Z}/\eta\) given in Proposition \ref{eta}. The readers refer to \S\ref{SH} for the definition of \(\mathcal{SH}(k)\) and \(H_W\mathbb{Z}\). Here it suffices to know that
\[H^{p,q}_W(X)=H^{p-q}(X,\textbf{I}^q)\]
if \(p\geq 2q-1\) by \cite[Proposition 4.5]{Y1} (similarly, \(H^{p,q}_M(X)=H^{p-q}(X,\textbf{K}_q^M/2), p\geq 2q-1\)). From this we have
\[H_W^{p,q}(X)=H^{p-q}(X,\textbf{W})\]
if \(p>2q\) and
\[H^{0,0}_W(X)=H^0(X,\textbf{W}).\]
If \(cd_2(k)<\infty\) the spectral sequence strongly converges to \(H^{p-q}(X,\textbf{W})\) (Proposition \ref{upper} and Proposition \ref{finite}).

Let us focus on the case when \(dim(X)=1\). By \ref{key} we have \(H^i(X,\textbf{W})=W^i(X),i=0\sim3\). By the norm residue theorem, we can identify the motivic cohomologies \(H^{*,*}_M(X),p\in\mathbb{Z},q\in\mathbb{N}\) as following :
\[H^{p,q}_M(X)=\begin{cases}H^p(X)\cdot\tau^{q-p}&p\leq q\\0&p>q+1\end{cases}.\]
Otherwise there is an exact sequence
\[0\to H^{q+1,q}_M(X)\to H^{q+1}(X)\to\mathcal{H}^{q+1}(X)\to0.\]
So the \(\tau\) is a non-zerodivisor on \(H^{*,*}_M\). By \cite[Proposition 2.7]{FY}, we have
\[E(W)_2^{p,q}=\frac{Ker(Sq^2) (\textrm{mod }\tau)}{Im(Sq^2) (\textrm{mod }\tau)}.\]

We say that the BSS has up to \(\eta^r\)-torsions (\cite[Definition 2.2]{FY}) if \(Ker(\eta^r)=Ker(\eta^{r+1})\) in \(H^{*,*}_W(X)\).
\begin{lemma}
The BSS has up to \(\eta\)-torsions if and only if \(Ker(\pi\circ\partial)=Ker(\partial)\).
\end{lemma}
\begin{proof}
If \(Ker(\eta)=Ker(\eta^2)\) and \(\pi\partial(x)=0\), there is a \(y\) such that \(\partial(x)=\eta y\). So \(\eta\partial(x)=\eta^2y=0\) by exact couple hence \(\eta y=0=\partial(x)\) by the condition. Conversely, if \(Ker(\pi\circ\partial)=Ker(\partial)\) and \(\eta^2x=0\), there is a \(y\) such that \(\eta x=\partial(y)\). So \(\pi(\eta x)=\pi\partial(y)=0\) by exact couple hence \(\partial(y)=0=\eta x\) by the condition.
\end{proof}

We denote by \(\beta^r:E(W)_r^{p,q}\to E(W)_r^{p+r+1,q+r}\) the \(r^{th}\) Bocksteins.
\begin{proposition}\label{upper}
If \(cd_2(k)=r<\infty\) and \(X\in Sm/k\), the BSS of \(X\) is strongly convergent. Moreover, the (BSS of) \(X\) has up to \(\eta^{r+d_X+1}\)-torsions.
\end{proposition}
\begin{proof}
We consider the filtration \(F^qH^i(X,\textbf{W})=Im(H^{q+i,q}_W(X)\to H^i(X,\textbf{W}))\) by multiplying \(\eta^r\) for large \(r\). Denote by \(H_W^{p,q}\) the sheafification of the functor \(X\longmapsto H_W^{p,q}(X)\). We have a Postnikov spectral sequence (\cite[Remark 4.3.9]{Mo1})
\[H^n(X,H_W^{p,q})\Longrightarrow H^{n+p,q}_W(X).\]
By \cite[Proposition 5.1]{AF}, \(H_W^{q,q}=\textbf{I}^q=0\) if \(q>r+d_X\). So \(H_W^{q,q}=H_{M/2}^{q,q}\) if \(q\geq r+d_X\). If \(p\neq q\), \(H_W^{p,q}=H_M^{p,q}\) by \cite[Theorem 17]{B} so we have
\[H^{p,q}_W(X)=H^{p,q}_M(X)\]
if \(q\geq r+d_X\). Hence for any \(a\in H^{p,q}_W(X)\) satisfying \(\eta^{r+d_X+2}a=0\), if \(q\leq r+d_X\) we have \(\eta^{r+d_X+1}a=0\) by degree reason. If \(q>r+d_X\), we have \(\eta a=0\) by the equality above. So the statement follows.
\end{proof}
\begin{proposition}\label{finite}
If \(cd_2(k)<\infty\), for any \(X\in Sm/k\), the (BSS of) \(X\) has up to \(\eta^r\)-torsions if and only if the BSS degenerates on \(E^{r+1}\)-page.
\end{proposition}
\begin{proof}
The necessity comes from \cite[Proposition 2.3]{FY}. For the sufficiency, we prove by descending induction on \(r\). If \(r\) is large enough, apply Proposition \ref{upper}. Suppose \(\eta^{r+1}(y)=0\) so \(\eta^r(y)=\partial(x)\) for some \(x\in E(W)_1^{*,*}\), \(\pi(y)=\pi(z)\) for some \(\eta^r(z)=0\). So \(y=z+\eta w\) thus \(\eta^{r+2}(w)=0\). By induction hypothesis, we have \(\eta^{r+1}(w)=0\) so \(\eta^r(y)=0\).
\end{proof}
\section{The Special Fiber}
Suppose that \(X\) is a proper curve over a perfect field \(k, char(k)\neq 2\) with \(H^0(X,O_X)=k\) and the normalization \(p:\widetilde{X^{red}}\to X\).
\begin{definition}
We define a reduced proper curve \(Y\) whose underlying space is the same as \(X\) and structure sheaf satisfies
\[
	\xymatrix
	{
		O_Y\ar[r]\ar[d]							&p_*O_{\widetilde{X^{red}}}\ar[d]\\
		\oplus_{x\in X_{sing}}k(x)\ar[r]&\oplus_{y\in p^{-1}(X_{sing})}k(y)
	},
\]
where \(X_{sing}\) is the set of singular points of \(X^{red}\). So the \(p\) decomposes as \(\widetilde{X^{red}}\xrightarrow{p_1}Y\xrightarrow{p_2}X\) where \(p_2\) is both a finite morphism and homeomorphism and \(Y\) has ordinary multiple points (\cite[Definition 5.13]{Liu}, \cite[pp. 247]{BLR}).
\end{definition}
\begin{definition}
Define \'etale sheaves
\[\begin{array}{cc}G=Ker(Pic^0_{X/k}\to Pic^0_{\widetilde{X^{red}}/k})&F=Coker(\mathbb{Z}/2\to p_*\mathbb{Z}/2).\end{array}\]
\end{definition}
\begin{definition}
Define
\[R=Im({_2}Pic(\widetilde{X^{red}})\xrightarrow{\partial} H^1(k,{_2}G)).\]
\end{definition}
\begin{remark}\label{Lang}
There is an exact sequence
\[0\to R\to H^1(k,{_2}G)\to H^1(k,{_2}Pic^0(X))\to H^1(k,{_2}Pic^0(\widetilde{X^{red}})).\]
By Lang's theorem \cite[Theorem 20.3]{St} we have \(H^1(k,Pic^0_{X/k})=0\). So by Kummer sequence we have
\[H^1(k,{_2}Pic(X))=Pic^0(X)/2,\]
which gives
\[H^1(k,{_2}G)/R=Ker(Pic^0(X)/2\to Pic^0(\widetilde{X^{red}})/2).\]
\end{remark}
\begin{remark}\label{div}
The \(G\) is a smooth connected linear algebraic group by \cite[Corollary 11, \S 9.2]{BLR}. So it is isomorphic to \(U\times T\) where \(U\) is unipotent and \(T\) is a torus by \cite[Theorem 2, \S 9.2]{BLR}. So \(T\) is \(2\)-divisible \'etale locally and \(U\) is uniquely \(2\)-divisible (\cite[Proposition 15.23]{M1}). Consequently \(G\) is \(2\)-divisible \'etale locally.
\end{remark}
\begin{lemma}\label{nil}
We have
\[H^*(X)=H^*(X^{red})\]
\[\begin{array}{cc}{_2}Pic(X)={_2}Pic(Y)={_2}Pic(X^{red})&Pic(X)/2=Pic(Y)/2=Pic(X^{red})/2.\end{array}\]
\end{lemma}
\begin{proof}
For any scheme \(X\), the \(X\) and \(X^{red}\) are homeomorphic hence they have the same \(\mathbb{Z}/2\). So the first equality follows. By \cite[Proposition 5 and Proposition 9, \S 9.2]{BLR}, the kernel of \(\pi:Pic_{X/k}\to Pic_{X^{red}/k}\) and \(\pi':Pic_{X^{red}/k}\to Pic_{Y/k}\) are unipotent, hence uniquely \(2\)-divisible. Apply \(H^*(k,\pi), H^*(k,\pi')\) (Remark \ref{Lang}) we get the last two equalities.
\end{proof}
\begin{proposition}\label{finite field}
We have
\[\begin{array}{cc}H^2(X)=Pic(X)/2&H^3(X)=H^3(\widetilde{X^{red}})=H^0(\widetilde{X^{red}})^{\vee}\end{array}\]
\[dim(H^1(X))=dim(Pic(X)/2)-dim(H^0(\widetilde{X^{red}}))+dim(H^0(X)).\]
\end{proposition}
\begin{proof}
By Lemma \ref{nil} we may suppose \(X\) is reduced. We have an exact sequence
\[H^i(\tilde{X})\to H^i(X,F)\to H^{i+1}(X)\to H^{i+1}(\tilde{X})\to H^{i+1}(X,F).\]
Suppose \(x_1,\cdots,x_r\in X\) are singular points of \(X\). We have
\[H^i(X,F)=\oplus_{j=1}^rH^i(k(x_j),F|_{x_j})\]
and a long exact sequence
\begin{equation}\label{C}H^i(k(x_j))\to H^i(p^{-1}(x_j))\to H^i(k(x_j),F|_{x_j})\to H^{i+1}(k(x_j))\end{equation}
by localization and proper base change. So by duality (\cite[Corollary 2.3, V]{M}) and additivity we have (\(\chi(-,2)=\sum_i(-1)^idim(H^i(-))\))
\[\chi(D,2)=\sum_j\chi(p^{-1}(x_j),2)-\chi(k(x_j),2)=0\]
\[0=\chi(\tilde{X},2)=\chi(X,2)+\chi(F,2)=\chi(X,2).\]
From this we obtain
\[dim(H^2(X))-dim(H^1(X))=dim(H^3(X))-dim(H^0(X)).\]
On the other hand, by duality we have
\[H^3(X)=H^3(\tilde{X})=H^0(\tilde{X})^{\vee}\]
and by \cite[7.2]{CTS} and \cite[2.15]{G} we have
\[Br(\tilde{X})=Br(X)=Br(k(x_j))=0.\]
So we get equalities desired.
\end{proof}
\begin{remark}
By Proposition \ref{finite field} we get
\[dim(H^2(X_k))=dim(Pic(\widetilde{X^{red}})/2)+dim(H^1(k,{_2}G)/R).\]
\end{remark}
\begin{proposition}\label{picard}
Denote by \(Z(x)=Ker(H^1(x)\to\oplus_{p(y)=x}H^1(y))\) for \(x\in X\). There are exact sequences
\[H^0(\widetilde{X^{red}})\to\frac{\oplus_{y\in p^{-1}(X_{sing})}H^0(y)}{\oplus_{x\in X_{sing}}H^0(x)}\to H^0(k,{_2}G)\to\oplus_{x\in X_{sing}}Z(x)\to0\]
\[O^{\times}(\widetilde{X^{red}})/2\to\frac{\oplus_{x'\in p^{-1}(X_{sing})}H^1(x')}{\oplus_{x\in X_{sing}}H^1(x)}\to H^1(k,{_2}G)\to0.\]
\end{proposition}
\begin{proof}
Denote by \(f\) (resp. \(\tilde{f}\)) the structure map of \(X\) (resp. \(\widetilde{X^{red}}\)) to \(Spec(k)\). By \cite[pp. 203]{BLR}, we have
\[\begin{array}{cc}Pic_{X/k}=R^1f_*\mathbb{G}_m&Pic_{\widetilde{X^{red}}/k}=R^1\tilde{f}_*\mathbb{G}_m\end{array}\]
as \'etale sheaves. By same statements as in \cite[2.10]{G} we have
\[\begin{array}{cc}H^n(X,\mathbb{G}_m)=H^{n-1}(k,Pic_{X/k})&H^n(\widetilde{X^{red}},\mathbb{G}_m)=H^{n-1}(k,Pic_{\widetilde{X^{red}}/k}),n\geq1\end{array}.\]
By Remark \ref{div} we may suppose \(X\) is reduced. In this case, we have \(R^2f_*\mathbb{Z}/2=\mathbb{Z}/2\) and \(R^2f_*\mathbb{G}_m\)=0 by \cite[7.2]{CTS}. The Kummer sequence gives an exact sequence (same for \(\tilde{f}\))
\[0\to R^1f_*\mathbb{Z}/2\to R^1f_*\mathbb{G}_m\to R^1f_*\mathbb{G}_m\to\mathbb{Z}/2\to 0.\]
Hence we have
\[\begin{array}{cc}{_2}Pic(X)=H^0(k,R^1f_*\mathbb{Z}/2)&{_2}Pic(\widetilde{X^{red}})=H^0(k,R^1\tilde{f}_*\mathbb{Z}/2)\end{array}.\]
Applying \(Rf_*\) to the exact sequence defining \(F\) we have an exact sequence
\[0\to f_*\mathbb{Z}/2\to\tilde{f}_*\mathbb{Z}/2\to f_*F\to R^1f_*\mathbb{Z}/2\to R^1\tilde{f}_*\mathbb{Z}/2.\]
So we obtain an exact sequence
\begin{equation}\label{gal}0\to f_*\mathbb{Z}/2\to\tilde{f}_*\mathbb{Z}/2\to f_*F\to{_2}G\to 0.\end{equation}

The \(H^0(k,{_2}G)\) is identified with \(Ker({_2}Pic^0(X)\xrightarrow{p^*}{_2}Pic^0(\widetilde{X^{red}}))\), which is the cokernel of the map
\[H^0(\widetilde{X^{red}})\to H^0(X,F).\]
The first result follows from the exact sequence
\[\oplus_{x\in X_{sing}}H^0(x)\to\oplus_{y\in p^{-1}(X_{sing})}H^0(y)\to H^0(X,F)\to\oplus_{x\in X_{sing}}Z(x)\to0.\]

For the second statement, suppose \(H^0(\widetilde{X^{red}},O_{\widetilde{X^{red}}})\) is a direct sum of finite fields \(k_a\). We have
\[\begin{array}{cc}H^i(k,f_*\mathbb{Z}/2)=H^i(k)&H^i(k,\tilde{f}_*\mathbb{Z}/2)=\oplus_aH^i(k,\mathbb{Z}/2|_{k_a})=\oplus_aH^i(k_a)\end{array}.\]
The terms above vanish when \(i\geq 2\). Hence we get an exact sequence by \eqref{gal}
\begin{equation}\label{gal1}\oplus_aH^1(k_a)\to H^1(k,f_*F)\to H^1(k,{_2}G)\to 0.\end{equation}
Since \(F\) is a direct sum of skycraper sheaves we have \(R^1f_*F=0\), hence \(H^1(k,f_*F)=H^1(X,F)\) by Leray spectral sequence. The statement follows from \eqref{C}, \eqref{gal1} and \(H^1(k_a)=O^{\times}(k_a)/2\).
\end{proof}

If \(X\) is smooth, there is a Merkurjev's pairing \(<-,->\) (\cite[5.1]{Sch1})
\[\begin{array}{ccccc}Pic(X)/2&\times&\mathcal{H}^1(X)&\to&K_1^M(k)/2=H^1(k)\\x&,&f&\mapsto&Tr_{k(x)/k}(f|_{k(x)})\end{array}.\]
We have
\[\mathcal{H}^1(X)=H^0(X,\textbf{K}_1^M/2)=H^{1,1}_M(X)=H^1(X).\]
\begin{lemma}
The pairing is defined for a general proper curve \(X\), namely
\[\begin{array}{ccccc}CH_0(X)/2&\times&H^1(X)&\to&H^1(k)\\x&,&f&\mapsto&Tr_{k(x)/k}(f|_{k(x)})\end{array}.\]
\end{lemma}
\begin{proof}
We can define a pairing \(<-,->\)
\[\begin{array}{ccccc}Z_0(X)/2&\times&H^1(X)&\to&H^1(k)\\x&,&f&\mapsto&Tr_{k(x)/k}(f|_{k(x)})\end{array}.\]
For any \(f\in H^1(X)\) and \(x\in\widetilde{X^{red}}\), we have
\begin{align*}
	&<x,p^*f>\\
=	&Tr_{k(x)/k}(p^*(f)|_{k(x)})\\
=	&Tr_{k(x)/k}(p^*(f|_{k(p(x))}))\\
=	&Tr_{k(p(x))/k}\circ Tr_{k(x)/k(p(x))}(p^*(f|_{k(p(x))}))\\
=	&[k(x):k(p(x))]Tr_{k(p(x))/k}(f|_{k(p(x))})\\
=	&<p_*x,f>.
\end{align*}
Hence for any \(g\in K(X)\) and \(f\in H^1(X)\), we have
\[<div(g),f>=<p_*div(p^*(g)),f>=<div(p^*(g)),p^*(f)>=0.\]
So the \(<-,->\) is well defined.
\end{proof}
When \(k\) is a finite field, we have \(H^1(k)=\mathbb{Z}/2\). The pairing above induces a map
\[CH_0(X)/2\xrightarrow{m}H^1(X)^{\vee}.\]
\begin{definition}\label{Merkurjev}
In the context above, define
\[\begin{array}{cc}S(X)=Coker(m)&D(X)=Im(m)\end{array}.\]
\end{definition}
\begin{proposition}\label{unramified}
Denote by \(\{x_i\}=X_{sing}\) and \(\{y_j\}=p^{-1}(X_{sing})\). We have
\[S(X)^{\vee}=Coker(H^0(\widetilde{X^{red}})\to\frac{\oplus_jH^0(y_j)}{\oplus_iH^0(x_i)}).\]
\end{proposition}
\begin{proof}
There is an exact sequence
\[\oplus_iH^*(x_i)\to\oplus_jH^*(y_j)\to H^*(X,F)\to\oplus_iH^{*+1}(x_i)\to\oplus_jH^{*+1}(y_j).\]
Taking dual there is an exact sequence
\[0\to H^1(X,F)^{\vee}\to\oplus_jH^1(y_j)^{\vee}\to\oplus_iH^1(x_i)^{\vee},\]
where the last arrow is identified with the pushforward of cycles along the normalization. We have a commutative diagram with exact rows
\[
	\xymatrix
	{
		K(\widetilde{X^{red}})^{\times}\ar[r]\ar@{=}[d]	&Z_0(\widetilde{X^{red}})/2\ar[r]\ar[d]^{a}	&CH_0(\widetilde{X^{red}})/2\ar[r]\ar[d]^{p_*}	&0\\
		K(X^{red})^{\times}\ar[r]											&Z_0(X)/2\ar[r]												&CH_0(X)/2\ar[r]												&0\\
	}
\]
where \(Ker(a)=H^1(X,F)^{\vee}\). Hence there is a surjection \(H^1(X,F)^{\vee}\to Ker(p_*)\) by Snake Lemma. From this we see that there is a commutative diagram with exact rows
\[
	\xymatrix
	{
		H^1(X,F)^{\vee}\ar[r]\ar@{=}[d]	&CH_0(\widetilde{X^{red}})/2\ar[r]^{p_*}\ar[d]_{\cong}	&CH_0(X)/2\ar[r]\ar[d]				&E\ar[d]_{\beta}\ar[r]	&0\ar[d]\\
		H^1(X,F)^{\vee}\ar[r]				&H^1(\widetilde{X^{red}})^{\vee}\ar[r]					&H^1(X)^{\vee}\ar[r]^{\gamma}		&H^0(X,F)^{\vee}\ar[r]^{\alpha}&H^0(\widetilde{X^{red}})^{\vee}
	}.
\]
We have
\[S(X)=Im(\gamma)/Im(\beta)\]
by Five Lemma. The \(Im(\beta)\) is generated by the image of singularities \(x_i\) so we have
\[Im(\gamma)^{\vee}=Ker(\alpha)^{\vee}=Coker(\alpha^{\vee})\]
\[Im(\beta)^{\vee}=Coker(\oplus_jH^0(y_j)\to H^0(X,F)).\]
So the statement follows from Snake Lemma.
\end{proof}

\begin{definition}\label{G1}
By Proposition \ref{picard}, define \(G(X)\) to be the image of the composite
\[\oplus_{y\in p^{-1}(X_{sing}),\sqrt{-1}\notin k(y)}H^1(y)\to H^1(k,{_2}G)\to H^1(k,{_2}Pic(X)).\]
\end{definition}
Denote by \(X\sqrt{-1}=X\times_kk\sqrt{-1}\). We have an exact sequence
\[\cdots\to H^i(X)\to H^i(X\sqrt{-1})\to H^i(X)\xrightarrow{\rho}H^{i+1}(X)\to\cdots.\]
\begin{proposition}\label{D}
We have an exact sequence
\[0\to Ker(D(X)\to D(X\sqrt{-1}))\to\rho^{\vee}H^2(X)^{\vee}\to G(X)^{\vee}\to0.\]
\end{proposition}
\begin{proof}

Define \(T(X)=Ker(S(X)\to S(X\sqrt{-1}))\). There is a commutative diagram with exact rows
\[
	\xymatrix
	{
		0\ar[r]	&D(X)\ar[r]\ar[d]		&H^1(X)^{\vee}\ar[r]\ar[d]		&S(X)\ar[r]\ar[d]		&0\\
		0\ar[r]	&D(X\sqrt{-1})\ar[r]	&H^1(X\sqrt{-1})^{\vee}\ar[r]	&S(X\sqrt{-1})\ar[r]	&0\\
	}.
\]
By Snake Lemma there is an exact sequence
\[0\to Ker(D(X)\to D(X\sqrt{-1}))\to\rho^{\vee} H^2(X)^{\vee}\xrightarrow{\varphi}T(X).\]
There is a commutative diagram
\[
	\xymatrix
	{
		\oplus_jH^0(y_j)\ar@{->>}[r]^c\ar@{->>}[dd]_{\rho}	&S(X)^{\vee}\ar[r]\ar@{->>}[d]			&H^1(X)\ar[d]\\
																					&T(X)^{\vee}\ar[r]^{\varphi^{\vee}}	&\rho H^1(X)\ar@{^{(}->}[d]\\
		\oplus_{y\in p^{-1}(X_{sing}),\sqrt{-1}\notin k(y)}H^1(y)\ar[r]^-{a}						&H^1(X,F)\ar[r]^b									&H^2(X)
	}
\]
where \(c\) is given by Proposition \ref{unramified}. So we have
\[Im(\varphi^{\vee})=Im(b\circ a)=Im(\oplus_{y\in u(X)}H^1(y)\to Ker(Pic(X)/2\to Pic(\widetilde{X^{red}})/2))=G(X).\]
Hence the statement follows.
\end{proof}
\begin{definition}
For any Gorenstein scheme \(X\), define
\[\Theta(X)=\{L\in Pic(X)|L\otimes L=\omega_X\}\]
to be the set of Theta characteristic of \(X\).
\end{definition}
Let us suppose in the sequel that the \(X\) is a connected Gorenstein projective curve over \(k\) with a rational point.
\begin{lemma}\label{fixed}
Suppose \(C\in\Theta(X_{\bar{k}})\) and denote by \(\sigma\in Gal(\bar{k}/k)\) the Frobenius. The \(\Theta(X)\neq\emptyset\) if and only if the map
\[\begin{array}{ccc}f:{_2Pic(X_{\bar{k}})}&\to&_2Pic(X_{\bar{k}})\\L&\mapsto&\sigma(L\otimes C)\otimes C^{-1}\end{array}\]
has a fixed point.
\end{lemma}
\begin{proof}
The \(Pic^0_{X}\) is represented by a scheme of locally finite type over \(k\) by \cite[Theorem 3,\S8.2]{BLR}, hence so does \(_2Pic_{X}\). So every \(L\in{_2Pic(X_{\bar{k}})}\) is defined over a finite extension \(k'/k\). Hence the \(Gal(\bar{k}/k)\) acts continuously on \(_2Pic(X_{\bar{k}})\). The \(f\) has a fixed point \(M\) if and only if \(C\otimes M\) is fixed by \(\sigma\), which is equivalent to \(C\otimes M\) is defined over \(k\) and is a square root of \(\omega_{X/k}\) by continuity of the Galois action.
\end{proof}
The following result was proven in \cite[Theorem 4.3]{R}.
\begin{lemma}\label{sqsm}
We have \(\Theta(\widetilde{X^{red}})\neq\emptyset\).
\end{lemma}
\begin{proof}
The Weil pairing on \(_2Pic(\widetilde{X_{\bar{k}}^{red}})\) is nondegenerate and the \(f\) defined in Lemma \ref{fixed} preserves the Weil pairing. Apply \cite[Remarks, pp. 61]{A}.
\end{proof}
\begin{lemma}\label{simplify}
For any finite extension \(k'/k\) with odd degree, the \(\Theta(X)\neq\emptyset\) is equivalent to \(\Theta(X_{k'})\neq\emptyset\).
\end{lemma}
\begin{proof}
By Proposition \ref{finite field}, the map \(Pic(X)/2\to Pic(X_{k'})/2\) is injective by the same statement of \(H^2\).
\end{proof}

Suppose \(C\in\Theta(X_{\bar{k}})\). The theta-form (\cite[1.12]{Harr})
\[q(L)=dim(H^0(X_{\bar{k}},L\otimes C))-dim(H^0(X_{\bar{k}},C))(\textrm{mod }2)\]
defined on \(_2Pic(X_{\bar{k}})\) is quadratic form and induces the Weil pairing \(<-,->\) on \(H^1(X_{\bar{k}})\), which is obtained by the pairing after pulling back to the normalization (\cite[Theorem 1.13, Proposition 2.2]{Harr}). Let \(f\) be the affine transformation defined in Lemma \ref{fixed}, which satisfies
\[q(f(L))=q(L).\]
The \(q|_{_2G}\) is a linear function on \(_2G\) and is independent of the choice of \(C\) (\cite[Proposition 2.3,Proposition 2.4]{Harr}), denoted by \(l\). Note that \(l=l\circ\sigma\) on \(_2G\) since the Weil pairing vanishes on \(_2G\). So \(l\) is a linear function on \(H^1(k,{_2}G)\).
\begin{proposition}\label{sq1}
The followings are equivalent:
\begin{enumerate}
\item\label{a} The \(p^*\omega_{X_k}\) is a square;
\item\label{b} There exists some \(C\in\Theta(X_{\bar{k}})\) satisfying \(p^*(\sigma(C)\otimes C^{-1})=0\);
\item\label{c} \(\Theta(X_{\bar{k}})\neq\emptyset\) and \(l=0\) on \(R\);
\end{enumerate}
which are necessary conditions of \(\Theta(X)\neq\emptyset\).

If moreover \((\sigma-Id)({_2}G)=(\sigma^2-Id)({_2}G)\), we have
\[\Theta(X_{k'})\neq\emptyset,[k':k]=2.\]
\end{proposition}
\begin{proof}
If \(p^*\omega_{X}=2D\), take \(p^*C=D\) for some \(C\in Pic(X_{\bar{k}})\). Hence \(2C=\omega_{X}+N\) where \(N\in G\). Since \(G(\bar{k})\) is \(2\)-divisible by Remark \ref{div}, we have \(N=2N',N'\in G\) in \(Pic(X_{\bar{k}})\). Then \(C-N'\in\Theta(X_{\bar{k}})\) and \(p^*(C-N')\) is Frobenius invariant. Conversely, if \(2C=\omega_{X_{\bar{k}}}\) and \(p^*C\) is Frobenius invariant, the \(p^*C\) is a square root of \(p^*\omega_{X}\). So we have proved (\ref{a})\(\Leftrightarrow\)(\ref{b}).

Let us prove (\ref{b})\(\Leftrightarrow\)(\ref{c}). Take any \(C\in\Theta(X_{\bar{k}})\). The (\ref{b}) is equivalent to
\begin{equation}\label{r}p^*((\sigma-Id)(C))=(\sigma-Id)(D)\end{equation}
for some \(2\)-torsion \(D\) by the surjectivity of \(p^*\). By the same reasoning as \cite[Lemma 5.1]{A}, we have
\[q(L)=q(\sigma(L))+<\sigma(L),f(0)>=q(\sigma(L))+<\sigma(L),(\sigma-Id)(C)>.\]
The pairing is nondegenerate on \(_2Pic(\widetilde{X_{\bar{k}}^{red}})\) with null space \(_2G\). By \textit{loc. cit.}, the \(\sigma\) admits an adjoint \(\sigma^*\) so that \(Id=\sigma^*\sigma\) on the normalization. Moreover we have
\[Im(\sigma-Id)=Ker(\sigma^*-Id)^{\perp}\]
on \({_2}Pic(\widetilde{X_{\bar{k}}^{red}})\). For any \(x\in Ker(\sigma^*-Id)=Ker(\sigma-Id)\), write \(x=p^*(y),y\in{_2}Pic(X_{\bar{k}})\). We have
\[<p^*((\sigma-Id)(C)),x>=<(\sigma-Id)(C),y>=q(y)-q(\sigma(y))=l((\sigma-Id)(y)).\]
So the \eqref{r} is equivalent to \(l=0\) on \(G(\bar{k})\cap(\sigma-Id)({_2}Pic(X_{\bar{k}}))\), which is the kernel of \(H^1(k,{_2}G)\to H^1(k,{_2}Pic(X_{\bar{k}}))\), being equal to \(R\).

For the last statement, we have \((\sigma-Id)^2(C)=(\sigma^2-Id)(C)\) since \(2\sigma(C)=2C\). So
\[(\sigma^2-Id)(C)\in(\sigma-Id)({_2}G)=(\sigma^2-Id)({_2}G)\]
by the condition. The \(Gal(\bar{k}/k')\) is generated by \(\sigma^2\) so the proof is finished.
\end{proof}
For every \(x\in\widetilde{X^{red}}\) and \(f\in O_{\widetilde{X^{red}},x}^{\times}\), we can define
\[f(x)=\prod_{y\in Spec(k(x)\times_k\bar{k})}f(y)\in H^1(x)=H^1(k)=\mathbb{Z}/2\]
where the Frobenius \(\sigma\) acts as \(\sigma\cdot f=\sigma\circ f\circ\sigma^{-1}\). This induces the Tate pairing
\[[-,-]:\begin{array}{ccc}{_2}Pic(\widetilde{X^{red}})\times Pic^0(\widetilde{X^{red}})/2&\to&\mathbb{Z}/2\\(D,E)&\mapsto&f(E)\end{array}\]
where \(div(f)=2D\) and \(Supp(f)\cap Supp(E)=\emptyset\) (\cite{Br}, our notation avoids the extension degree).

For any component \(X_a\subseteq{\widetilde{X^{red}}}\), the image of the map \(deg:Pic(X_a)\to\mathbb{Z}\) is the ideal generated by \([O(X_a):k]\) by Hasse-Weil inequality. One can take a divisor \(D\) on \(X_a\) such that \(deg(D)=[O(X_a):k]\) and \(Supp(D)\cap p^{-1}(X_{sing})=\emptyset\).

The \(deg:Div(\widetilde{X^{red}})\to\mathbb{Z}^{\oplus n}\), where \(n=dim(H^0(\widetilde{X^{red}}))\), has image ideal \(([O(X_a):k])_{a=1\sim n}\). So one can find an \(x_0\in Div(\widetilde{X^{red}})\) such that \(deg(x_0)=([O(X_a):k])\) and \(Supp(x_0)\cap p^{-1}(X_{sing})=\emptyset\).
\begin{proposition}\label{semisimple}
In the context above, the boundary map
\[\partial:{_2}Pic(\widetilde{X^{red}})\to H^1(k,{_2}G)\]
is given by
\[\partial(D)=\overline{\oplus_{x\in X_{sing},i}[D,x_i-\frac{deg(x_i)}{deg(x_0)}x_0]}.\]
\end{proposition}
\begin{proof}
We may suppose \(X=Y\) since we only consider \(2\)-torsion part. Then \(Pic(X)=CaCl(X)\). As \(p(Supp(x_0))\) lives in the smooth locus, for any \(f\in K(\widetilde{X^{red}})\) being inverible on \(Supp(x_0)\), the constant \(f(x_0)\) vanishes on \(H^1(k,{_2}G)\). Suppose the \(D\) is written as a Cartier divisor \((U_j,f_j)\) where either \(f_j=1\) and \(U_j\cap p^{-1}(X_{sing})=\{x\}\), or \(U_j\subseteq p^{-1}(X_{reg})\). Then for every \(x\in X_{sing}\), pick a small neighbourhood \(x\in V_x\) such that \(p^{-1}(V_x)\subseteq\cup_{x\in p(U_i)}U_i\). Then the divisor
\[D'=\{(V_x,1),(U_j,f_j)_{U_j\cap p^{-1}((X_k)_{sing})=\emptyset}\}\]
satisfies \(p^*(D')=D\) in \(Div(\widetilde{X^{red}})\). Take a \(v\in K(\widetilde{X_{\bar{k}}^{red}})\) such that for every \(x_i^t\in Spec(k(x_i)\times_k\bar{k}),t=0,\cdots,n-1\),
\[v(x_i^t)=\sqrt{f(x_i^t)}\]
where \(2D=div(f)\) is given in the condition. The \(f\) satisfies \(f\circ\sigma=\sigma\circ f\). Suppose each \(x_i^t\) can be written as \(\sigma^a(x_i^{t_0})\) for some \(t_0\), hence there is an equality
\begin{equation}\label{period}\sigma^n(x_i^{t_0})=x_i^{t_0}.\end{equation}

Let us set \(v(x_i^{t_0})\) to be any square root of \(\frac{f(x_i^{t_0})}{f(x_0)}\) in \(\bar{k}\). If we can choose these \(v(x_i^t)\) such that
\begin{equation}\label{period1}v(\sigma(x_i^t))=\sigma(v(x_i^t))\end{equation}
for every \(t\), then by \eqref{period} the \(v(x_i^{t_0})\) is fixed by \(\sigma^n\) hence \(v(x_i^{t_0})\in k(x_i)\), namely the \(\frac{f(x_i^{t_0})}{f(x_0)}\) is a square in \(k(x_i)\). The converse is clearly true. In case \(f(x_i^{t_0})\) is not a square, the \eqref{period1} is possible to be true simutaneously for all \(t\) such that \(\sigma(x_i^t)\neq x_i^{t_0}\).

Let us take a \(v\) satisfying all conditions discussed in the paragraph above. Shrinking \(V_x\) the divisor (on \(X_{\bar{k}}\))
\[D''=\{(V_x,v),(U_j,f_j)_{U_j\cap p^{-1}(X_{sing})=\emptyset}\}\]
satisfies
\[p^*(D'')=D,2D''=div(f)\]
on \(X_{\bar{k}}\). Then
\[(\sigma-Id)(D'')=\{(V_x,\frac{\sigma(v)}{v}),(U_j,1)_{U_j\cap p^{-1}(X_{sing})=\emptyset}\}.\]
Since we assumed \(G\) is a torus, this divisor has the representative
\[\overline{\oplus_i\frac{\sigma(v)}{v}(x_i)}\in H^1(k,{_2}G)\]
given in Proposition \ref{picard}. Using the notation before, we have
\[\frac{\sigma(v)}{v}(x_i)=\begin{cases}-1&f(x_i^{t_0})\notin (k(x_i)^{\times})^2\\1&f(x_i^{t_0})\in (k(x_i)^{\times})^2\end{cases},\]
which gives the statement.
\end{proof}
The Proposition \ref{semisimple} describes the extension of \(Pic^0_{\widetilde{X^{red}}/k}\) by \(G\) by Tate pairing. 
\begin{definition}\label{lambda}
Suppose \(L,M\in Div(\widetilde{X^{red}})\) have supports contained in \(X_{reg}\) and
\[L=2M+div(f).\]
Define
\[\Lambda(L)=(f(x))_{x\in p^{-1}(X_{sing})}\in H^1(k,{_2}G)/R\]
which is independent of the choice of \(M\).

If \(L\in Div(X)\), \(Supp(L|_{X^{red}})\subseteq X_{reg}\) and \(L|_{\widetilde{X^{red}}}\) is a square, we define \(\Lambda(L)=\Lambda(p^*L)\).
\end{definition}
\begin{proposition}\label{square}
Suppose that \(L\in Div(X)\), \(Supp(L|_{X^{red}})\subseteq X_{reg}\) and that \(p^*O(L)\) is a square. The \(O(L)\) is a square if and only if \(\Lambda(L)=0\).

In other words, the \(\Lambda\) computes the isomorphism
\[Ker(Pic^0(X)/2\to Pic^0(\widetilde{X^{red}})/2)\xrightarrow{\cong}H^1(k,{_2}G)/R.\]
\end{proposition}
\begin{proof}
We may suppose \(X\) is reduced. If \(O(L)\) is a square, we have \(p^*L=2p^*M+p^*div(f)\) so \(\Lambda(p^*L)=0\). Conversely, since \(p^*O(L)\) is a square, we have
\[L=2M+N\]
as Cartier divisors with supports being contained in \(X_{reg}\), where \(p^*N=div(f)\). The representative of \(N\) in \(H^1(k,{_2}G)=H^0(k,G)/2\) is \(\overline{\oplus_if(x_i)}\). If \(\Lambda(L)=0\), that representative goes to zero in \(Pic^0(X)/2\), namely \(N\) is a square. Hence \(O(L)\) is a square.
\end{proof}
\begin{definition}\label{poly}
For every \(\lambda\in\mathbb{P}^1_k\), define \(P_{\lambda}(x)\) to be the monic irreducible polynomial in \(k[x]\) such that
\[div(P_{\lambda}(x))=\lambda-deg(\lambda)\cdot\infty.\]
This gives a homomorphism \(Div(\mathbb{P}^1_k)\to k[x]\).
\end{definition}
\begin{definition}
Suppose that \(X\) is smooth with a finite map \(\pi:X\to\mathbb{P}^1\) such that \(\pi^*O(1)\) is a square. Define \(f_{\pi}\in K(X)\) by
\[\pi^*(\lambda)=div(\frac{P_{\lambda}(x)\circ\pi}{f_{\pi}})+2D\]
for some \(f_{\pi}\) (take some \(\pi^*(-\infty)=div(f_{\pi})+2E\)) being independent of \(\lambda\), which is unique up to some \(g\in K(X)^{\times}\) such that \(div(g)\in 2Div(X)\).
\end{definition}
\begin{remark}\label{factor1}
In the context above, suppose \(\pi\) is equal to a composite
\[X\xrightarrow{\pi'}\mathbb{P}^1\xrightarrow{\varphi}\mathbb{P}^1\]
where \(f_{\pi'}\) is defined. For any \(\lambda\in\mathbb{A}^1\), we have
\[\pi^*\lambda=\pi'^*\varphi^*\lambda=div(\frac{P_{\varphi^*\lambda}(x)\circ\pi'}{f_{\pi'}})+2D'.\]
If \(P_{\varphi^*\lambda},f_{\pi'}\) are inverible at \(x_0\in X\), we have
\[\frac{P_{\varphi^*\lambda}\circ\pi'}{f_{\pi'}}(x_0)=\frac{P_{\lambda}(\pi(x_0))}{f_{\pi'}(x_0)}.\]
\end{remark}
\section{The Generic Fiber}
Suppose that \(K\) is a local field with \(2\in O_K^{\times}\) and residue field \(k\) and that \(X\) is a regular proper flat curve on \(O_K\) with smooth generic fiber \(X_K\) and special fiber \(X_k\). Suppose that \(X_K\) is connected and has a rational point. It specializes to a rational point on \(X_k\), so \(tr(X_k)>0\).
\begin{proposition}\label{duality}
We have
\[H^i(X_K)=H^{4-i}(X_K)^{\vee}\]
as \(\mathbb{Z}/2\)-vector spaces.
\end{proposition}
\begin{proof}
By \cite{L} for every \(i\) there is an isomorphism
\[H^i(X_K,\mathbb{G}_m)\otimes_{\mathbb{Z}}\hat{\mathbb{Z}}=Hom(H^{3-i}(X_K,\mathbb{G}_m),\mathbb{Q}/\mathbb{Z}).\]
So
\[H^i(X_K,\mathbb{G}_m)/2=Hom(H^{3-i}(X_K,\mathbb{G}_m),\mathbb{Q}/\mathbb{Z})/2=({_2H}^{3-i}(X_K,\mathbb{G}_m))^{\vee}.\]
The statement follows from the exact sequence
\[0\to H^{i-1}(X_K,\mathbb{G}_m)/2\to H^i(X_K)\to{_2H}^i(X_K,\mathbb{G}_m)\to0.\]
\end{proof}
\begin{proposition}\label{etale}
We have
\begin{enumerate}
\item\(dim(H^0(X_K))=dim(H^0(X_k));\)
\item We have an exact sequence
\[0\to H^1(X_k)\to H^1(X_K)\to H^0(\widetilde{X_k^{red}})\to 0.\]
Consequencely \(dim(H^1(X_K))=dim(H^2(X_k))+1\).
\item\(dim(H^2(X_K))=2dim(Pic(X_k)/2);\)
\item\(Pic(X_K)/2=Pic(X)/2=Pic(X_k)/2=H^2(X_k)\).
\end{enumerate}
\end{proposition}
\begin{proof}
The map \(Pic(X)\to Pic(X_K)\) and \(Pic(X_K)\xrightarrow{deg}\mathbb{Z}\) are surjective so the map \(Pic^0(X)\to Pic^0(X_K)\) is surjective. The \(Pic^0_{X/O_K}\) is separated by \cite[Corollary 3, 9.4]{BLR} so we have \(Pic^0(X)=Pic^0(X_K)\), whence \(Pic(X)/2=Pic(X_K)/2\). 

By deformation theory every line bundle on \(X_k\) can be lifted to a line bundle on the formal scheme \(\widehat{X_{X_k}}\), which comes from a line bundle on \(X\) by Grothendieck existence theorem. So the map \(Pic(X)/2\to Pic(X_k)/2\) is surjective. But
\[Pic(X)/2\subseteq H^2(X)=H^2(X_k)=Pic(X_k)/2\]
by Proposition \ref{finite field} and Kummer sequence, which gives \(Pic(X)/2=Pic(X_k)/2\). Hence (4) is proved.

Denote by \(sing\) (resp. \(reg\)) the singular (resp. smooth) locus of \(X_k\), by \(X'=X\setminus sing\) and by \(p:\widetilde{X_k^{red}}\to X_k\) the normalization map. We have two Gysin sequences by absolute purity (\cite{T})
\[H^{i-4}(sing)\to H^i(X)\to H^i(X')\to H^{i-3}(sing)\]
\[H^{i-2}(reg)\to H^i(X')\to H^i(X_K)\to H^{i-1}(reg)\]
where \(H^*(X)=H^*(X_k)\) by proper base change. Then we have
\[H^0(X)=H^0(X')=H^0(X_K)\]
which gives (1).

By Kummer sequence we have
\[0\to Pic(X_K)/2\to H^2(X_K)\to{_2}H^2(X_K,\mathbb{G}_m)\to 0.\]
By the proof of Proposition \ref{duality} we have \({_2}H^2(X_K,\mathbb{G}_m)=(Pic(X_K)/2)^{\vee}\) so
\[dim(H^2(X_K))=2dim(Pic(X_K)/2)=2dim(Pic(X_k)/2),\]
proving (3). By additivity we have
\[\chi(X_K,2)=2dim(H^0(X_K))-2dim(H^1(X_K))+dim(H^2(X_K))\]
\[\chi(X_K,2)+\chi(reg,2)=\chi(X',2)\]
\[\begin{array}{cc}\chi(reg,2)+\chi(p^{-1}(sing),2)=\chi(\widetilde{X_k^{red}},2)&\chi(X',2)+\chi(sing,2)=\chi(X,2)=\chi(X_k,2)\end{array}\]
\[\begin{array}{cc}\chi(sing,2)=\chi(p^{-1}(sing),2)=0&\chi(X_k,2)=\chi(\widetilde{X_k^{red}},2)=0\end{array}\]
where the last equality follows from the proof of Proposition \ref{finite field}. This shows that
\[\begin{array}{cc}\chi(X_K,2)=0&dim(H^2(X_K))=2dim(H^1(X_K))-2dim(H^0(X_k))\end{array}.\]
We have an exact sequence by purity
\[0\to H^1(X_k)\to H^1(X_K)\to H^0(reg)=H^0(\widetilde{X_k^{red}})\]
and equalities
\begin{align*}
	&dim(H^1(X_K))\\
=	&dim(Pic(X_K)/2)+dim(H^0(X_k))\\
=	&dim(Pic(X_k)/2)+dim(H^0(X_k))\\
=	&dim(H^1(X_k))+dim(H^0(\widetilde{X_k^{red}}))\\
=	&dim(H^2(X_k))+1
\end{align*}
by (4) and Proposition \ref{finite field}. This shows that sequence is right exact, proving (2).
\end{proof}
\begin{remark}\label{residue}
We have exact sequences
\[H^2(K(X_K))\xrightarrow{q_1}\oplus_{x\in X_K^{(1)}}H^1(x)\to H^{3,2}_M(X_K)\to 0\]
\[\oplus_{y\in X_k^{(0)}}H^1(k(y))\xrightarrow{r_2}\oplus_{x\in X_k^{(1)}=X^{(2)}}H^0(k(x))\to CH_0(X_k)/2\to0\]
by \cite[Proposition 4.5]{Y1}. The composite of differentials of the Rost complex of \(X\)
\[H^2(K(X_K))\xrightarrow{q_1,q_2}\oplus_{x\in X_K^{(1)}}H^1(x)\oplus\oplus_{y\in X_k^{(0)}}H^1(k(y))\xrightarrow{r_1+r_2}\oplus_{x\in X^{(2)}}H^0(k(x))\]
is zero so we obtain a map \(\partial:H^{3,2}_M(X_K)\to CH_0(X_k)/2\) which is surjective by \cite[Corollary 1.38,\S10]{Liu}.
\end{remark}
\begin{proposition}\label{unramified1}
\begin{enumerate}
\item We have
\[S(X_k)=\mathcal{H}^3(X_K).\]
Moreover, there is an exact sequence
\[0\to H^0(\widetilde{X_k^{red}})^{\vee}\to H^{3,2}_M(X_K)\to D(X_k)\to0.\]
\item If \(X_k\) is reduced, we have \(D(X_k)=CH_0(X_k)/2\).
\end{enumerate}
\end{proposition}
\begin{proof}
\begin{enumerate}
\item The first equality follows from \cite[Lemma 2.4]{K} and Proposition \ref{unramified}. For the last statement, we have a commutative diagram with exact rows (using \cite[Proposition 1.36,\S 10]{Liu} and \(H^3(X_K)=H^1(X_K)^{\vee}\))
\[
	\xymatrix
	{
		0\ar[r]	&L\ar[r]\ar[d]_h						&H^{3,2}_M(X_K)\ar[r]\ar[d]_i	&D(X_k)\ar[r]\ar@{^{(}->}[d]_j		&0\\
		0\ar[r]	&H^0(\widetilde{X_k^{red}})^{\vee}\ar[r]	&H^3(X_K)\ar[r]											&H^1(X_k)^{\vee}\ar[r]	&0\\
	}
\]
We have \(Coker(i)=Coker(j)=\mathcal{H}^3(X_K)\) by results above so \(h\) is an isomorphism by Snake Lemma.
\item For every irreducible component \(Y\) of \(X_k\), pick a smooth closed point \(y\in Y\) by generic smoothness of \(X_k\). Then there is an open neighbourhood \(y\in U\) such that the composite \(U\to X\to Spec(O_K)\) is smooth. Choose \(R\) to be the CDVR of the unramified extension of \(K\) with \(k(R)=k(y)\). By smoothness and Hensel lemma we obtain a lift \(Spec(R)\to U\) such that the composite \(Spec(R)\to U\to Spec(O_K)\) is induced by \(O_K\subseteq R\). The \(R\) satisfies \(H^1(R)=(R^{\times})/2\). By \cite[Corollary, \S 2, Chapter V]{Ser1}, the norm map
\[N:H^1(R)\to H^1(O_K)\]
is surjective (hence isomorphism) since \(K(R)/K\) is unramified. Denote by \(x\in X\) the image of the generic point of \(R\) and by \(R_x\) the DVR representing the closure of \(x\). The \(N\) factors through the norm \(H^1(R_x)\to H^1(O_K)\) hence the latter is an isomorphism. Hence the push forward \(H^1(R_x)\to H^3(X)\) is injective. So there is a commutative diagram of exact rows
\[
	\xymatrix
	{
		0\ar[r]& H^1(R_x)\ar[r]\ar[d]	&H^3(X)\ar[r]\ar[d]_{\cong}		&H^3(X\setminus Spec(R_x))\ar[d]\\
		0\ar[r]& H^1(k(y))\ar[r]			&H^3(\widetilde{X_k^{red}})\ar[r]	&H^3(\widetilde{X_k^{red}}\setminus y)\\
	},
\]
which shows that the map \(\oplus_x H^1(R_x)\to H^3(X)\) is surjective. This is because \(H^3(\widetilde{X_k^{red}})\) is freely generated by the pushforward of \(H^1(y)\) for every component \(Y\) and some closed point \(y\in Y\).

By Proposition \ref{etale}, (2) and Propostion \ref{duality}, we have a commutative diagram with exact rows
\[
	\xymatrix
	{
		0\ar[r]	&L'\ar[r]\ar[d]						&H^{3,2}_M(X_K)\ar[r]^{\partial}\ar[d]_a	&CH_0(X_k)/2\ar[r]\ar[d]_b	&0\\
		0\ar[r]	&H^0(\widetilde{X_k^{red}})^{\vee}\ar[r] 	&H^3(X_K)\ar[r]						&H^1(X_k)^{\vee}\ar[r]			&0
	}.
\]
The \(L'\subseteq L=H^3(X)\) since \(D(X_k)\) is a quotient of \(CH_0(X_k)/2\). By discussion above we see that
\[H^3(X)\subseteq H^{3,2}_M(X_K).\]
Now take a point \(x\in X_K^{(1)}\) as above. There is a commutative diagram
\[
	\xymatrix
	{
		H^1(R_x)\ar[r]\ar[d]	&H^1(K(R_x))\ar[r]\ar[d]				&H^0(k(R_x))\ar[d]\\
		H^3(X)\ar[r]				&H^{3,2}_M(X_K)\ar[r]	&CH_0(X_k)/2
	}
\]
where the first row is a complex by purity. Since \(H^3(X)\) is generated by pushforwards of \(H^1(R_x)\), the second row is also a complex. Hence \(L'=H^3(X)\) and \(D(X_k)=CH_0(X_k)/2\).
\end{enumerate}
\end{proof}
By degree reason, the \(Sq^2\) (resp. \(\beta^2\)) on \(H^{p,q}_M(X_K)\) (resp. \(E(W)^{p,q}_2(X_K)\)) vanishes if \(p>2\) (resp. \(p>1\)) and the \(\beta^r=0\) if \(r\geq3\). If \(p<q\), we have \(H^{p,q}_M(X_K)=Im(\tau)\) so \(Sq^2\tau=\tau Sq^2+\rho\tau Sq^1\) and \(E_2(W)^{p,q}(X_K)=0\). If \(p>q, p\leq 2\), the \(Sq^2=\beta^2=0\) trivially. So to compute \(Sq^2\) (resp. \(\beta^2\)), it suffices to consider when \(p=q=2\) (resp. \(p=q=1\)).

\begin{proposition}\label{sq}
In the context above, the followings are equivalent:
\begin{enumerate}
\item The \(Sq^2:H^{2,2}_M(X_K)\to H^{4,3}_M(X_K)=\mathbb{Z}/2\) vanishes;
\item \(\Theta(X_K)\neq\emptyset\);
\item \(\Theta(X_k)\neq\emptyset\).
\end{enumerate}
\end{proposition}
\begin{proof}
We have a commutative diagram
\[
	\xymatrix
	{
		X_k\ar[r]^u\ar[d]_{q}	&X\ar[d]_{p}	&X_K\ar[l]_v\ar[d]_{\pi}\\
		Spec(k)\ar[r]				&Spec(O_K)		&Spec(K)\ar[l]
	}.
\]
The \(\Sigma^{\infty}(X_K)_+\) has a strong dual \(Th(\omega_{X_K/K})\wedge S^{-2,-1}\) (\(\mathcal{SH}(K)\)) by \cite{Ay}. The composite
\[S^{0,0}\xrightarrow{\delta}\Sigma^{\infty}(X_K)_+\wedge Th(\omega_{X_K/K})\wedge S^{-2,-1}\xrightarrow{\pi_*}Th(\omega_{X_K/K})\wedge S^{-2,-1}\]
realizes the push-forward \(\pi_*\) on \(H^{*,*}_W\) and \(H^{*,*}_M\), where \(\delta\) is the counit map. It commutes with the boundary map \(H_W\mathbb{Z}/\eta\to H_W\mathbb{Z}/\eta\wedge\mathbb{P}^1\) of the BSS (\S\ref{bockstein}), where differentials are explicitly given. So we obtain a commutative diagram
\[
	\xymatrix
	{
		H^{2,2}_M(X_K)\ar[r]^{Sq^2_{\omega_{X_K/K}}}\ar[d]_{\pi_*}	&H^{4,3}_M(X_K)\ar[d]_{\pi_*}\\
		H^{0,1}_M(K)\ar[r]^{Sq^2}											&H^{2,2}_M(K)
	}
\]
where \(Sq^2_{\omega_{X_K/K}}=Sq^2+\omega_{X_K/K}\). The \(Sq^2=0\) on \(H^{0,1}_M(K)\) and the \(\pi_*\) induces an isomorphism \(H^{4,3}_M(X_K)=H^{2,2}_M(K,\mathbb{Z}/2)\) so we find that
\[Sq^2=\omega_{X_K/K}\]
on \(H^{2,2}_M(X_K)=H^2(X_K)\). The Proposition \ref{duality} gives \((1)\Leftrightarrow(2)\).

The \(X_k\) is a locally complete intersection hence they are Gorenstein. Thus the dualizing sheaf \(\omega_{X/O_K}=p^!O_K[-1]\) is invertible. It satisfies
\[\begin{array}{cc}v^*\omega_{X/O_K}=\omega_{X_K/K}&u^*\omega_{X/O_K}=q^!k[-1]\end{array}\]
by \cite[\S 19, Duality for Schemes]{SP}. So \(\omega_{X_K/K}\) is a square if and only if \(q^!k[-1]\) is a square by Proposition \ref{etale}.
\end{proof}
\begin{proposition}\label{ac}
Suppose \(\{Y_i\}\) are irreducible components of \(X_{\bar{k}}\) with multiplicity \(d_i\). The \(\Theta(X_{\bar{k}})\neq\emptyset\) if and only if the number
\[-Y_i^2=\frac{\sum_{j\neq i}d_jY_j\cdot Y_i}{d_i}\]
is even for all \(i\).
\end{proposition}
\begin{proof}
There is an isomorphism
\[Pic(X_{\bar{k}})/2\cong\oplus_iPic(Y_i)/2\]
by the same way as in Proposition \ref{finite} since \(Pic/2=H^2\) over \(\bar{k}\). Set \(X'=X\times_{O_K}O_{K^{ur}}\) where \(K^{ur}\) is the maximal unramified extension of \(K\). By \cite[\S 19, Duality for Schemes]{SP}, we have equalities (\(Y_i\) are Gorenstein)
\[\omega_{X_{\bar{k}}}|_{Y_i}=\omega_{X'/O_{K^{ur}}}|_{Y_i}=\omega_{Y_i/\bar{k}}\otimes O_{X'}(-Y_i)|_{Y_i}.\]
We have \(Pic(Y_i/\bar{k})/2=\mathbb{Z}/2\) by counting degrees and
\[deg(\omega_{Y_i/\bar{k}}\otimes O_{X'}(-Y_i)|_{Y_i})\equiv -Y_i^2=\frac{\sum_{j\neq i}d_jY_j\cdot Y_i}{d_i}(\textrm{mod }2)\]
by \cite[(32), page 72]{Ser}, \cite[Theorem 1.37, \S 9]{Liu} and \cite[Proposition 1.21, \S 9]{Liu}. So the statement follows.
\end{proof}

\begin{lemma}\label{modrho}
Denote by \(H^{*,*}_{M/\rho}(-)\) the modulo \(\rho\) version of \(H^{*,*}_M(-)\). For any curve \(X\in Sm/K\), we have
\[H^{p,q}_{M/\rho}(X)=\begin{cases}H^p(X\sqrt{-1})&p<q\\H^{*,*}_M(-)/\rho&p\geq q+1\end{cases}.\]
If \(p=q\) there is an inclusion \(H^{p,q}_{M/\rho}(X)\subseteq H^p(X\sqrt{-1})\) with cokernel \(M\), satisfying an exact sequence
\[0\to M\to\mathcal{H}^p(X)\xrightarrow{\rho}\mathcal{H}^{p+1}(X).\]
\end{lemma}
\begin{proof}
Follows from the commutative diagram with long exact rows
\[
	\xymatrix
	{
		\cdots\ar[r]&H^{p,q}_M(X)\ar[r]\ar[d]	&H^{p,q}_{M/\rho}(X)\ar[r]\ar[d]	&H^{p,q-1}_M(X)\ar[r]^-{\rho}\ar[d]	&\cdots\\
		\cdots\ar[r]&H^p(X)\ar[r]										&H^p(X\sqrt{-1})\ar[r]					&H^{p}(X)\ar[r]^-{\rho}										&\cdots
	}
\]
and the Five Lemma.
\end{proof}
\begin{lemma}\label{rho}
Suppose \(\sqrt{-1}\notin K\). We have
\[dim(\rho Pic(X_K)/2)=tr(X_k).\]
\end{lemma}
\begin{proof}
By Proposition \ref{etale}, (2) (and its dual), there is a commutative diagram
\[
	\xymatrix
	{
		Pic(X_K)/2\ar@{^{(}->}[r]\ar[d]_{\rho}	&H^2(X_K)\ar[d]_{\rho}\\
		H^3(X_k)\ar@{^{(}->}[r]							&H^3(X_K)
	}.
\]
So \(\rho Pic(X_K)/2=\rho H^2(X_k)=\rho Pic(X_k)/2\). Denote by \(Y_i\) the irreducible components of \(\widetilde{X_k^{red}}\). The map \(deg:Pic(Y_i)/2\to\mathbb{Z}/2\) is surjective if and only if \([O(Y_i):k]\) is odd by Hasse-Weil inequality. Then the statement follows from the commutative diagram
\[
	\xymatrix
	{
		H^2(X_k)\ar[r]^{\rho}\ar@{->>}[d]				&H^3(X_k)\ar[d]_{\cong}\\
		\oplus_iH^2(Y_i)\ar[r]^{\rho}\ar[d]_{Tr}	&\oplus_iH^3(Y_i)\ar[d]_{Tr}\\
		\oplus_iH^0(k)\ar[r]^{\rho}						&\oplus_iH^1(k)
	}
\]
\end{proof}
Denote by \(l(M)\) the length function for any Artinian module \(M\).
\begin{definition}
Define
\[q=dim(Coker(H^{2,2}_M(X_K)\xrightarrow{Sq^2}H^{4,3}_M(X_K)))\]
\[q'=dim(Coker(H^{2,2}_M(X_{K\sqrt{-1}})\xrightarrow{Sq^2}H^{4,3}_M(X_{K\sqrt{-1}}))).\]
\end{definition}
\begin{theorem}\label{h1}
Suppose we have a smooth proper curve \(X_K\) over a nondyadic local field \(K\) with a rational point.
\begin{enumerate}
\item\label{431} We have
\[H^{4,3}_W(X_K)=\mathbb{Z}/2.\]
Moreover, the \(H^{*,*}_W(X_K)\) has up to \(\eta\)-torsions.
\item\label{432}
We have
\[dim(2H^{3,2}_W(X_K))=\begin{cases}q&\sqrt{-1}\notin K\\0&\sqrt{-1}\in K\end{cases},\]
\[l(H^{3,2}_W(X_K))=q+t\]
where \(t=dim(H^0(\widetilde{X_k^{red}}))+dim(H^1(X_k))-dim(S(X_k))=dim(H^{3,2}_M(X_K))\).
\item We have
\[dim(2W^1(X_K))=\begin{cases}q+tr(X_k)&\sqrt{-1}\notin K\\0&\sqrt{-1}\in K\end{cases}\]
\[l(W^1(X_K))=q+t+dim(H^2(X_k)).\]
\end{enumerate}
\end{theorem}
\begin{proof}
\begin{enumerate}
\item By Proposition \ref{upper} we have
\[H^{4,3}_W(X_K)=H^{4,3}_M(X_K),\]
where the latter is a subgroup of \(H^4(X_K)\). 

Any Galois extension \(L/K\) solvable so it is a composite of extension of prime degrees. The \(Tr_{L/K}\) induces an isomorphism on \(K_2^M/2\) when \([L:K]\) is odd or equal to \(2\) by \cite[Theorem 40.3]{EKM}. Hence it's an isomorphism for any finite extension. So for every component \(Y\) of \(X_K\), the push-forward map \(H^{4,3}_M(X_K)\to K_2^M(K)/2\) is surjective. The \(H^{4,3}_M(X_K)\) is a subgroup of \(H^4_{\textrm{\'et}}(X_K,\mathbb{Z}/2)=\mathbb{Z}/2\). So \(H^{4,3}_M(X_K)=\mathbb{Z}/2\). This implies that
\[H^{4,3}_M(X_K)=\mathbb{Z}/2.\]

For the last statement, it suffices to prove \(\beta^2=0\) on \(E_2(W)^{1,1}(X_K)=H^{1,1}_M(X_K)\) when \(s=1\) by \(cd_2(K)=2\), Proposition \ref{upper} and Proposition \ref{finite}. If \(Sq^2\neq 0\), we have \(E_2(W)^{4,3}(X_K)=0\) so \(\beta^2=0\). Otherwise we have \(\omega_{X_K/K}=0\in Pic(X_K)/2\) by Proposition \ref{sq} so there is a commutative diagram
\[
	\xymatrix
	{
		H^{4,3}_W(X_K)\ar[r]^{\eta}\ar[d]_{\pi_*}^{\cong}	&H^{3,2}_W(X_K)\ar[d]_{\pi_*}\\
		I^2(K)\ar@{^(->}[r]^{\eta}										&I(K)
	}
\]
So the upper horizontal map is injective. Hence the vanishing of \(\beta^2\) is equivalent to that of the composite
\[H^{1,1}_M(X_K)\xrightarrow{\partial} H^{3,2}_W(X_K)\xrightarrow{\pi_*}I(K).\]
This gives the claim since \(\partial\circ\pi_*=\pi_*\circ\partial\) and \(\pi_*=0\) on \(H^{1,1}_M\).
\item 
We have
\[dim(H^{3,2}_M(X_K))+dim(\mathcal{H}^3(X_K))=dim(H^0(\widetilde{X_k^{red}}))+dim(H^1(X_k))\]
by Proposition \ref{duality} and Proposition \ref{etale}. Combining with Proposition \ref{unramified}, we get
\[dim(H^{3,2}_M(X_K))=dim(H^0(\widetilde{X_k^{red}}))+dim(H^1(X_k))-dim(S(X_k)).\]
There is an exact sequence
\[H^{2,2}_M(X_K)\xrightarrow{Sq^2}H^{4,3}_W(X_K)\xrightarrow{\eta} H^{3,2}_W(X_K)\to H^{3,2}_M(X_K)\to0\]
and a commutative diagram
\[
	\xymatrix
	{
		H^{3,2}_W(X_K)\ar[r]^{\rho}\ar@{->>}[d]	&H^{4,3}_W(X_K)\ar[d]_{\cong}\\
		H^{3,2}_M(X_K)\ar[r]^{\rho}\ar@{->>}[d]	&H^{4,3}_M(X_K)\ar[d]_{\cong}\\
		H^1(K)\ar@{->>}[r]^{\rho}						&H^2(K)
	},
\]
hence the first horizonal map is surjective. So \(dim(2H^{3,2}_W(X_K))=q\) by \(2=\rho\eta\), whereas \(l(H^{3,2}_W(X_K))=q+t\).
\item By \eqref{key} it suffices to compute \(H^1(X_K,\textbf{W})\). The \(H^1(X_K,\textbf{W})=H^{2,1}_W(X_K)\) follows from the surjectivity of \(\textbf{W}(X_K)\xrightarrow{rk}\mathbb{Z}/2\). We have exact sequences
\[0\to H^{3,2}_W(X_K)\xrightarrow{\eta} H^{2,1}_W(X_K)\to Pic(X_K)/2\to0\]
\[0\to2H^{2,1}_W(X_K)\to H^{3,2}_W(X_K)\to H^{3,2}_{W/\rho}(X_K)\to 0.\]
The \(\eta\) is injective because \(Sq^2=0\) on \(H^{1,1}(X_K,\mathbb{Z}/2)\) and \(Ker(\eta)\cong Im(Sq^2)\) by \(\eta\)-torsion property established in (1).

There is an exact sequence
\[H^{2,2}_{M/\rho}(X_K)\to H^{4,3}_{W/\rho}(X_K)\to H^{3,2}_{W/\rho}(X_K)\to H^{3,2}_{M/\rho}(X_K)\to0.\]
We have \(H^{4,3}_{W/\rho}(X_K)=H^{4,3}_M(X_K)/\rho=0\) by discussion above. This implies that
\begin{align*}
	&dim(2H^{2,1}_W(X_K))\\
=	&l(H^{3,2}_{W}(X_K))-dim(H^{3,2}_{W/\rho}(X_K))\\
=	&q+dim(\rho Pic(X_K)/2).
\end{align*}
Furthermore we have
\begin{align*}
	&l(H^{2,1}_W(X_K))\\
=	&l(H^{3,2}_W(X_K))+dim(Pic(X_K)/2)\\
=	&q+t+dim(Pic(X_K)/2)\\
=	&q+dim(H^0(\widetilde{X_k^{red}}))+dim(H^1(X_k))-dim(S(X_k))+dim(H^2(X_k))\\
=	&q+2dim(H^2(X_k))+1-dim(S(X_k)).
\end{align*}
Hence the statement follows.
\end{enumerate}
\end{proof}
\begin{lemma}\label{rhorho}
We have
\[\rho H^1(X_K)\cap\tau Pic(X_K)/2=\rho H^1(X_k).\]
\end{lemma}
\begin{proof}
There is a commutative diagram with exact rows and columns
\begin{equation}\label{rhoh1}
	\xymatrix
	{
		0\ar[r]	&H^1(X_k)\ar[r]\ar[d]_{\rho}			&H^1(X_K)\ar[r]\ar[d]_{\rho}			&H^0(\widetilde{X_k^{red}})\ar[r]\ar[d]_{\alpha^{\vee}}	&0\\
		0\ar[r]	&H^2(X_k)\ar[r]\ar[d]			&H^2(X_K)\ar[r]\ar[d]			&Pic(X_K)/2^{\vee}\ar[r]								&0\\
		0\ar[r]	&H^2(X_{k\sqrt{-1}})\ar[r]	&H^2(X_{K\sqrt{-1}})			&																	&
	}
\end{equation}
and an exact sequence
\[0\to\frac{\rho H^1(X_K)}{\rho H^1(X_K)\cap\tau Pic(X_K)/2}\to\frac{H^2(X_K)}{\tau Pic(X_K)/2}\to\frac{H^2(X_K)}{\rho H^1(X_K)+\tau Pic(X_K)/2}\to0.\]
By Snake Lemma the diagram \eqref{rhoh1} gives an exact sequence
\begin{equation}\label{alpha}0\to H^2(X_k)/\rho H^1(X_k)\to H^2(X_K)/\rho H^1(X_K)\to Coker(\alpha^{\vee})\to0\end{equation}
We have
\[\frac{H^2(X_K)}{\tau Pic(X_K)/2}=Pic(X_K)/2^{\vee}\]
\[\frac{H^2(X_K)}{\rho H^1(X_K)+\tau Pic(X_K)/2}=Coker(\alpha^{\vee})=Ker(\alpha)^{\vee},\]
hence
\[dim(\frac{\rho H^1(X_K)}{\rho H^1(X_K)\cap\tau Pic(X_K)/2})=dim(Pic(X_K)/2)-dim(Ker(\alpha)).\]
Finally by \eqref{alpha} we have
\begin{align*}
	&dim(\rho H^1(X_K)\cap\tau Pic(X_K)/2)\\
=	&dim(\rho H^1(X_K))-dim(Pic(X_K)/2)+dim(Ker(\alpha))\\
=	&dim(Pic(X_K)/2)+dim(\rho H^1(X_k))-dim(Ker(\alpha))-dim(Pic(X_K)/2)+dim(Ker(\alpha))\\
=	&dim(\rho H^1(X_k)).
\end{align*}
\end{proof}
\begin{lemma}\label{dsq}
Suppose \(q'=1\). Denote by \(d:H^3(X_K)\to H^1(X_k)^{\vee}\) the boundary map obtained by taking dual of the map in Proposition \ref{etale}, (2). The followings are equivalent:
\begin{enumerate}
\item\(D(X_k)\cap\rho^{\vee}H^2(X_k)^{\vee}=D(X_k)\cap d\rho(Ker(Sq^2)^{2,2})\);
\item\(\omega_{X_k}\in G(X_k)\).
\end{enumerate}
\end{lemma}
\begin{proof}
If \(q=1\) the statement is clear. Suppose \(q=0\). The cup product pairing on \(H^2(X_K)\) induces a commutative diagram
\[
	\xymatrix
	{
		H^2(X_K)\ar[r]\ar[d]_{Sq^2}	&H^2(X_k)^{\vee}\ar[ld]^{ev_{\omega_{X_k}}}\\
		\mathbb{Z}/2							&
	}.
\]
The condition shows that the composite
\[H^2(X_{k\sqrt{-1}})^{\vee}\to H^2(X_k)^{\vee}\xrightarrow{ev_{\omega_{X_k}}}\mathbb{Z}/2\]
is zero so we obtain a map
\[\overline{ev_{\omega_{X_k}}}:{\rho}H^2(X_k)^{\vee}\to\mathbb{Z}/2.\]
Denote by \(T(X_k)=Ker(S(X_k)\to S(X_{k\sqrt{-1}}))\). There is a commutative diagram
\[
	\xymatrix
	{
		Ker(\overline{ev_{\omega_{X_k}}})\ar[r]^{i}	&\rho^{\vee} H^2(X_k)^{\vee}\ar[r]^{\varphi}\ar[d]	&T(X_k)\ar[d]\\
																			&H^1(X_k)^{\vee}\ar[r]									&S(X_k)
	}
\]
and maps
\[T(X_k)^{\vee}\xrightarrow{\varphi^{\vee}}\rho H^1(X_k)\xrightarrow{i^{\vee}}\rho H^1(X_k)/\omega_{X_k}\]
where \(Im(\varphi^{\vee})=G(X_k)\) by Proposition \ref{D}. Then the 1) is equivalent to
\[\omega_{X_k}\in G(X_k).\]
\end{proof}
Note that if \(\omega_{X_k}\in G(X_k)\), the \(p^*\omega_{X_k}\) must be a square.
\begin{definition}\label{delta}
Define
\[\delta=\begin{cases}0&q'=0\\0&q'=1, \omega_{X_k}\in G(X_k)\\1&\textrm{else}\end{cases}.\]
\end{definition}
\begin{theorem}\label{h0}
Given a smooth proper curve \(X_K\) over a nondyadic local field \(K\) with a rational point, we have the followings:
\begin{enumerate}
\item The \(4\)-torsion group \(H^{2,2}_W(X_K)\) satisfies
\[l(H^{2,2}_W(X_K))=dim(S(X_k))+2dim(H^2(X_k))-1+q\]
\[dim(2H^{2,2}_W(X_K))=\begin{cases}dim(G(X_k))+\delta&\sqrt{-1}\notin K\\0&\sqrt{-1}\in K\end{cases}.\]
\item The \(4\)-torsion group \(H^{1,1}_W(X_K)\) satisfies
\[l(H^{1,1}_W(X_K))=dim(S(X_k))+2dim(H^2(X_k))+q,\]
\[dim(2H^{1,1}_W(X_K))=\begin{cases}dim(G(X_k))+tr(X_k)+\delta&\sqrt{-1}\notin K\\0&\sqrt{-1}\in K\end{cases}.\]
\item The \(4\)-torsion group \(W(X_K)\) satisfies
\[l(W(X_K))=dim(S(X_k))+2dim(H^2(X_k))+q+1,\]
\[dim(2W(X_K))=\begin{cases}dim(G(X_k))+tr(X_k)+\delta+1&\sqrt{-1}\notin K\\0&\sqrt{-1}\in K\end{cases}.\]
\end{enumerate}
\end{theorem}
\begin{proof}
\begin{enumerate}
\item By Proposition \ref{upper}, we have that \(H^{p,p}_W(X_K)=H^p(X_K)\) if \(p\geq 3\). There is an exact sequence
\begin{equation}\label{rho2}E_1^{1,2}(W)(X_K)\to H^{3,3}_W(X_K)\to H^{2,2}_W(X_K)\to H^{2,2}_M(X_K)\xrightarrow{Sq^2} H^{4,3}_W(X_K).\end{equation}
The first map is written as
\[H^{1,1}_M(X_K)\oplus H^{3,2}_M(X_K)\xrightarrow{(\tau\rho Sq^1,\tau)}H^3(X_K)\]
by \cite[Lemma 9.9]{V}. So we obtain an exact sequence
\begin{equation}\label{rho6}0\to\mathcal{H}^3(X_K)\to H^{2,2}_W(X_K)\to Ker(Sq^2)^{2,2}\to0,\end{equation}
giving the first equation.

The sequence \eqref{rho2} shows that
\begin{equation}\label{G}2H^{2,2}_W(X_K)=\frac{\rho Ker(Sq^2)^{2,2}}{\rho Ker(Sq^2)^{2,2}\cap\tau H^{3,2}_M(X_K)}.\end{equation}
There is a commutative diagram with exact rows and columns
\begin{equation}\label{rho23}
	\xymatrix
	{
					&												&H^2(X_{K\sqrt{-1}})\ar[r]\ar[d]	&H^2(X_{k\sqrt{-1}})^{\vee}\ar[r]\ar[d]	&0\\
		0\ar[r]	&Pic(X_k)/2\ar[r]\ar[d]_{\alpha}&H^2(X_K)\ar[r]^{d'}\ar[d]_{\beta}	&H^2(X_k)^{\vee}\ar[r]\ar[d]	&0\\
		0\ar[r]	&H^0(\widetilde{X_k^{red}})^{\vee}\ar[r]	&H^3(X_K)\ar[r]				&H^1(X_k)^{\vee}\ar[r]			&0
	}
\end{equation}
where vertical arrows are given by multiplication by \(\rho\) and 
\[Im(\alpha)=\rho Pic(X_k)/2.\]

We consider the commutative diagram with exact rows by Proposition \ref{unramified}, (1) and the diagram \eqref{rho23}
\[\xymatrixcolsep{1pc}
	\xymatrix
	{
		0\ar[r]	&H^0(\widetilde{X_k^{red}})^{\vee}\ar[r]\ar@{->>}[d]		&H^{3,2}_M(X_K)\ar[r]\ar[d]_v	&D(X_k)\ar[r]\ar[d]_w										&0\\
		0\ar[r] &Coker(\alpha)\ar[r]^-{i}										&H^3(X_K)/\rho(Ker(Sq^2)^{2,2})\ar[r]						&H^1(X_k)^{\vee}/\rho^{\vee}d'(Ker(Sq^2)^{2,2})\ar[r]	&0
	},
\]
where the \(i\) is injective since the \(Coker(\alpha)\to Coker(\beta)\) is injective. So we obtain an exact sequence by Snake Lemma
\[0\to Im(\alpha)\to Ker(v)\to Ker(w)\to 0.\]
By Proposition \ref{D}, we have
\[dim(D(X_k)\cap\rho^{\vee} H^2(X_k)^{\vee})=dim(\rho^{\vee} H^2(X_k)^{\vee})-dim(G(X_k)).\]

If \(q'=0\), pick any \(a\in H^{2,2}_M(X_{K\sqrt{-1}})\) such that \(Sq^2(a)\neq 0\). The \(a'=Tr_{K\sqrt{-1}/K}(a)\) satisfies \(Sq^2(a')\neq 0\) and \(\rho\cdot a'=0\). So \(\rho H^2(X_K)=\rho Ker(Sq^2)^{2,2}\) if \(q'=0\). Combining with Lemma \ref{dsq} when \(q'=1\), we have
\begin{equation}\label{intersection}dim(Ker(w))=dim(\rho^{\vee} H^2(X_k)^{\vee})-dim(G(X_k))-\delta.\end{equation}
We have
\begin{align}
	&dim(2H^{2,2}_W(X_K))\\
=	&dim(\rho H^2(X_K))-dim(Ker(v))\\
=	&dim(\rho H^2(X_K))-dim(Ker(w))-dim(Im(\alpha))\\
=	&dim(\rho H^2(X_K))-dim(\rho^{\vee} H^2(X_k)^{\vee})+dim(G(X_k))-dim(Im(\alpha))+\delta\\
=	&dim(H^3(X_K))-dim(H^0(\widetilde{X_k^{red}})^{\vee})-dim(H^1(X_k)^{\vee})+dim(G(X_k))+\delta\\
=	&dim(G(X_k))+\delta\label{222}
\end{align}
\item The \(Sq^2\) vanishes on \(H^{1,1}\) and \(H^{*,*}_W\) has up to \(\eta\)-torsion on \(X_K\), so the map \(H^{1,1}_M(X_K)\to H^{3,2}_W(X_K)\) is zero. There is an exact sequence as in \eqref{rho2}
\begin{equation}\label{rho4}0\to Pic(X_K)/2\xrightarrow{\tau}H^{2,2}_W(X_K)\to H^{1,1}_W(X_K)\to H^1(X_K)\to0.\end{equation}
So we have equalities
\[2H^{1,1}_W(X_K)=\frac{\rho H^{1,1}_W(X_K)}{\rho H^{1,1}_W(X_K)\cap\tau Pic(X_K)/2}\]
\begin{align*}
	&l(H^{1,1}_W(X_K))\\
=	&l(H^{2,2}_W(X_K))+dim(H^1(X_K))-dim(Pic(X_K)/2)\\
=	&dim(S(X_k))+dim(Ker(Sq^2)^{2,2})+dim(H^1(X_K))-dim(Pic(X_K)/2)\\
=	&dim(S(X_k))+dim(H^2(X_k))-1+q+dim(H^1(X_k))+dim(H^0(\widetilde{X_k^{red}}))\\
=	&dim(S(X_k))+2dim(H^2(X_k))+q.
\end{align*}

There is a distinguished triangle in \(\mathcal{SH}(K)\)
\[H_W\mathbb{Z}/\rho\wedge\mathbb{G}_m\xrightarrow{\eta}H_W\mathbb{Z}/\rho\to H_{\mu}\mathbb{Z}/2/\rho\oplus H_{\mu}\mathbb{Z}/2/\rho[2],\]
which induces a long exact sequence by Lemma \ref{modrho}
\begin{equation}\label{22rho}H^{1,2}_{M/\rho}(X_K)\oplus H^{3,2}_{M/\rho}(X_K)\xrightarrow{(u,v)} H^{3,3}_{W/\rho}(X_K)\to H^{2,2}_{W/\rho}(X_K)\to H^{2,2}_{M/\rho}(X_K)\to H^{4,3}_{W/\rho}(X_K).\end{equation}

We have commutative diagrams
\[
	\xymatrix
	{
																								&H^{1,2}_{M/\rho}(X_K)\oplus H^{3,2}_{M/\rho}(X_K)\ar[d]_{(u,v)}\ar[r]	&H^{1,1}_M(X_K)\ar[ldd]^0\\
		H^{4,4}_{W/\rho}(X_K)=\mathbb{Z}/2\ar[r]^-{\eta}\ar@{=}[d]	&H^{3,3}_{W/\rho}(X_K)\ar[d]			&\\
		H^{4,3}_W(X_K)\ar[r]^-{\eta'}							&H^{3,2}_W(X_K)												&
	}
\]
\[
	\xymatrix
	{
		H^{3,2}_M(X_K)\ar@{->>}[r]\ar[d]_{\tau}			&H^{3,2}_{M/\rho}(X_K)\ar[d]_v\\
		H^3(X_K)/\rho Ker(Sq^2)^{2,2}\ar@{^{(}->}[r]	&H^{3,3}_{W/\rho}(X_K)
	},
\]
which yields an exact sequence
\[0\to\frac{H^3(X_K)}{Im(u)+\rho Ker(Sq^2)^{2,2}+\tau H^{3,2}_M(X_K)}\to Coker(u,v)\to{_\rho H^{3,2}_W(X_K)}\to 0.\]
Suppose that the first term has dimension
\[dim(S(X_k))-dim(G(X_k))-\delta',\]
where the \(dim(S(X_k))-dim(G(X_k))-\delta+dim(\rho H^2(X_K))-dim(\rho Ker(Sq^2)^{2,2})\) is the dimension of
\[\frac{H^3(X_K)}{\rho Ker(Sq^2)^{2,2}+\tau H^{3,2}(X_K)}\]
by \eqref{intersection}. The
\begin{equation}\label{same}Sq^2:H^{1,2}_M(X_{K\sqrt{-1}})\to H^{3,3}_M(X_{K\sqrt{-1}})\end{equation}
is zero because \(H^{1,2}=\tau H^{1,1}\), \(Sq^2\tau=\tau Sq^2+\rho\tau Sq^1=\tau Sq^2\) and \(Sq^2=0\) on \(H^{1,1}\). So \(Im(u)\subseteq Im(\eta)\subseteq\mathbb{Z}/2\) by Lemma \ref{modrho}. So we have
\begin{equation}\label{ineq}\delta-dim(\rho H^2(X_K))+dim(\rho Ker(Sq^2)^{2,2})\leq\delta'\leq\delta-dim(\rho H^2(X_K))+dim(\rho Ker(Sq^2)^{2,2})+1.\end{equation}

We have an exact sequence
\[0\to H^{3,3}_W(X_K)/\rho\to H^{3,3}_{W/\rho}(X_K)\to {_\rho H^{3,2}_W(X_K)}\to 0\]
where \(H^{3,3}_W(X_K)/\rho=H^3(X_K)/\rho Ker(Sq^2)^{2,2}\). Hence if \(q=q'\), by discussion above we have
\[\rho Ker(Sq^2)^{2,2}=\rho H^2(X_K)\]
so \(H^{3,3}_W(X_K)/\rho\subseteq H^3(X_{K\sqrt{-1}})\). Hence the \(u\) is identified with the zero map in \eqref{same}. Hence \(\delta=\delta'\) if \(q=q'\).

We have \(H^{4,3}_{W/\rho}(X_K)=H^{4,3}_M(X_K)/\rho=0\) by \(tr(X_K)>0\). So the sequence \eqref{22rho} becomes
\[0\to Coker(u,v)\to H^{2,2}_{W/\rho}(X_K)\to H^{2,2}_{M/\rho}(X_K)\to0.\]
Now we also have an exact sequence
\[0\to H^{2,2}_W(X_K)/\rho\to H^{2,2}_{W/\rho}(X_K)\to{_\rho H^{2,1}_W(X_K)}\to0.\]
By Theorem \ref{h1}, we have
\[dim(_\rho H^{2,1}_W(X_K))=dim(_2H^{2,1}_W(X_K))=t+dim(Pic(X_K)/2)-dim(\rho Pic(X_K)/2).\]
Hence we obtain
\begin{align}\footnotesize
	&dim(\rho H^{1,1}_W(X_K))\\
=	&l(H^{2,2}_W(X_K))-dim(H^{2,2}_W(X_K)/\rho)\\
=	&dim(S(X_k))+dim(Ker((Sq^2)^{2,2}))-dim(H^{2,2}_{W/\rho}(X_K))+dim(_\rho H^{2,1}_W(X_K))\label{22rho'}\\
=	&dim(Ker((Sq^2)^{2,2}))+dim(G(X_k))-dim(H^{2,2}_{M/\rho}(X_K))\\
	&-q+1+dim(Pic(X_K)/2)-dim(\rho Pic(X_K)/2)+\delta'\\
=	&dim(Ker((Sq^2)^{2,2}))+dim(G(X_k))-dim(H^2(X_K)/\rho H^1(X_K))-q+1+\delta'\\
=	&dim(G(X_k))+dim(\rho H^1(X_K))+\delta'\label{11rho}
\end{align}
By the same discussion as in \eqref{22rho}, there is an exact sequence
\[H^{0,1}_{M/\rho}(X_K)\oplus H^{2,1}_{M/\rho}(X_K)\xrightarrow{(\xi,\tau)}H^{2,2}_{W/\rho}(X_K)\xrightarrow{\eta} H^{1,1}_{W/\rho}(X_K).\]
We have
\[H^{0,1}_{M/\rho}(X_K)=\mathbb{Z}/2\cdot\tau,H^{2,1}_{M/\rho}(X_K)=Pic(X_K)/2\]
so \(\xi=0\) by \(Sq^2(\tau)=0\). Thus we obtain the exact sequence
\[Pic(X_K)/2\xrightarrow{\tau}H^{2,2}_{W/\rho}(X_K)\xrightarrow{\eta} H^{1,1}_{W/\rho}(X_K).\]
The kernel of the composite
\[Pic(X_K)/2\to H^{2,2}_W(X_K)\to H^{2,2}_{W/\rho}(X_K)\to H^{2,2}_{M/\rho}(X_K)\]
is the same as that of the composite
\[Pic(X_K)/2\to H^2(X_K)\to H^{2,2}_{M/\rho}(X_K).\]
This is just \(\rho H^1(X_K)\cap\tau Pic(X_K)/2\), which equals to \(\rho H^1(X_k)\) by Lemma \ref{rhorho}.

Now we want to study the kernel of the map
\[Pic(X_K)/2=H^{2,1}_{M/\rho}(X_K)\xrightarrow{\tau}H^{2,2}_{W/\rho}(X_K),\]
which is identified with the kernel of 
\[\rho H^{1,1}_W(X_K)\xrightarrow{\eta}2H^{1,1}_W(X_K).\]
There is a commutative diagram
\[\xymatrixcolsep{1pc}
	\xymatrix
	{
		H^{2,1}_{W/\rho}(X_K)\ar[rr]^{\tau'}\ar@{->>}[rd]\ar[dd]&														&H^{2,2}_{W/\rho}(X_K)\ar[dd]\\
																							&H^{2,1}_{M/\rho}(X_K)\ar[ru]^{\tau}	&\\
		H^{2,0}_W(X_K)=0\ar[rr]									&														&H^{2,1}_W(X_K)\\
	}.
\]
So the \(\tau\) factors through \(H^{2,2}_W(X_K,\mathbb{Z})/\rho\subseteq H^{2,2}_{W/\rho}(X_K)\). We have an exact sequence
\[0\to\rho H^{1,1}_W(X_K)/2H^{2,2}_W(X_K)\to H^{2,2}_W(X_K)/2\to H^{2,2}_W(X_K)/\rho\to0.\]
By \eqref{11rho} and \eqref{222} we see that
\[dim(\rho H^{1,1}_W(X_K)/2H^{2,2}_W(X_K))=dim(\rho H^1(X_K))+\delta'-\delta.\]
On the other hand, the composite
\begin{equation}\label{inj}\rho H^{1,1}_W(X_K)/2H^{2,2}_W(X_K)\to H^{2,2}_W(X_K)/2\xrightarrow{\pi}H^2(X_K)\end{equation}
has image \(\rho H^1(X_K)\), which shows that \(\delta'\geq\delta\). If \(q\neq q'\), we have
\[dim(\rho Ker(Sq^2)^{2,2})=dim(\rho H^2(X_K))-1\]
hence by \eqref{ineq} we obtain \(\delta'=\delta\) in any case. Hence the map \eqref{inj} is injective.

Now we have
\[Pic(X_K)/2\cap2H^{2,2}_W(X_K)=0\]
by the \(\eta\)-torsion property of \(H^{*,*}_W\). Hence the \(\pi\) is injective on \(Pic(X_K)/2\cap\rho H^{1,1}_W(X_K)\). So the kernel we wonder is just \(\tau Pic(X_K)/2\cap\rho H^1(X_K)\) by looking at its image under \(\pi\).

The \eqref{alpha} gives the equality
\[dim(\rho H^1(X_K))-dim(\rho H^1(X_k))=dim(Im(\alpha^{\vee})),\]
where \(\alpha^{\vee}\) is the map
\[Pic(X_k)/2\xrightarrow{\rho}H^0(X_k^{red})^{\vee}=H^3(X_k).\]
Then the \eqref{11rho} gives
\[dim(2H^{1,1}_W(X_K))=dim(G(X_k))+dim(\rho H^2(X_k))+\delta.\]
Finally \(dim(\rho H^2(X_k))=tr(X_k)\) was proved in the proof of Lemma \ref{rho}.
\item By \eqref{key} it suffices to compute \(H^0(X_K,\textbf{W})\). There is an exact sequence as before
\begin{equation}\label{++1}0\to H^{1,1}_W(X_K)\to H^{0,0}_W(X_K)\xrightarrow{rk}H^0(X_K)=\mathbb{Z}/2\to 0\end{equation}
so \(2H^{0,0}_W(X_K)=\rho H^{0,0}_W(X_K)\). We have exact sequences
\[0\to H^{1,1}_W(X_K)/\rho\to H^{1,1}_{W/\rho}(X_K)\to{_\rho}H^{1,0}_W(X_K)={_2H^{2,1}_W(X_K)}\to0\]
\[H^{0,1}_{M/\rho}(X_K)\oplus H^{2,1}_{M/\rho}(X_K)\xrightarrow{(0,v)} H^{2,2}_{W/\rho}(X_K)\to H^{1,1}_{W/\rho}(X_K)\to H^{1,1}_{M/\rho}(X_K)\xrightarrow{s} H^{3,2}_{W/\rho}(X_K).\]
We have 
\[\begin{array}{cc}H^{1,1}_{M/\rho}(X_K)=H^{1,1}_M(X_K)/\rho&H^{3,2}_{W/\rho}(X_K)=H^{3,2}_W(X_K)/\rho\end{array}\]
so \(s=0\) by the vanishing of \(H^{1,1}_M(X_K)\to H^{3,2}_W(X_K)\). So we obtain an exact sequence
\[0\to\frac{Pic(X_K)/2}{\rho H^1(X_K)\cap\tau Pic(X_K)/2}\to H^{2,2}_{W/\rho}(X_K)\to H^{1,1}_{W/\rho}(X_K)\to H^1(X_K)/\rho\to0.\]
By \eqref{22rho'} we have
\[dim(H^{2,2}_{W/\rho}(X_K))=dim(S(X_k))+dim(Ker(Sq^2)^{2,2})+dim(_2H^{2,1}_W(X_K))-dim(\rho H^{1,1}_W(X_K)).\]
Hence
\[dim(_2H^{2,1}_W(X_K))-dim(H^{2,2}_{W/\rho}(X_K))=dim(\rho H^{1,1}_W(X_K))-dim(S(X_k))-2dim(H^2(X_k))-q+1.\]
This shows that
\begin{align}
	&dim(2H^{0,0}_W(X_K))\\
=	&l(H^{1,1}_W(X_K))+dim(_2H^{2,1}_W(X_K))-dim(H^{1,1}_{W/\rho}(X_K))\\
=	&l(H^{1,1}_W(X_K))+dim(_2H^{2,1}_W(X_K))-dim(H^{2,2}_{W/\rho}(X_K))\\
	&-dim(H^1(X_K))+1+dim(H^2(X_k))-dim(\rho H^1(X_k))\\
=	&\label{+1}dim(2H^{1,1}_W(X_K))+1.
\end{align}
Finally apply the computation of \(H^{1,1}_W(X_K)\), the \eqref{+1} and \eqref{++1}.
\end{enumerate}
\end{proof}
\begin{remark}
Since \(H^{1,0}_M=H^{0,-1}_M=0\), we have
\begin{align*}
	&\sum_idim(\mathcal{H}^i(X_K))\\
=	&\sum_idim(H^i(X_K))-\sum_iH^{i+1,i}_M(X_K)\\
=	&\sum_{i\leq2}dim(H^i(X_K))+dim(S(X_k))-dim(Pic(X_K)/2)\\
=	&1+dim(H^2(X_k))+1+2dim(H^2(X_k))+dim(S(X_k))-dim(Pic(X_K)/2)\\
=	&2+2dim(H^2(X_k))+dim(S(X_k)).
\end{align*}
Hence we recover the statement in \cite{PSr}, namely
\[l(W(X_K))=\sum_idim(\mathcal{H}^i(X_K))+q-1.\]
\end{remark}
\section{Witt Group of Elliptic Curves}\label{elliptic}
Suppose that \(X_K\) is an elliptic curve with \(\sqrt{-1}\notin K\) and \(char(k)>3\). By \(\omega_{X_K}=O_{X_K}\) and Proposition \ref{sq} we find that \(q=q'=1\) so \(\delta=0\). By classification of reduction types of minimal model in \cite[\S 10.2]{Liu} and \cite[Theorem 8.2, IV]{Sil}, we get by Theorem \ref{h0} the following computation:
\[\textbf{W}(X_K)=\mathbb{Z}/4^{\oplus a}\oplus\mathbb{Z}/2^{\oplus b}\]
\[b=2dim(Pic(X_k)/2)+dim(S(X_k))-2dim(G(X_k))-2tr(X_k)\]
\[a=dim(G(X_k))+tr(X_k)+1.\]
In the case of good reduction, namely type \(I_0\), we have \(S(X_k)=G(X_k)=R=0\). Otherwise the \(\widetilde{X_k^{red}}\) is a disjoint of \(\mathbb{P}^1\)s, hence \(R=0\).

So in any case there is an exact sequence
\[O^{\times}(\widetilde{X_k^{red}})/2\xrightarrow{j}\frac{\oplus_{x'\in p^{-1}((X_k)_{sing})}H^1(x')}{\oplus_{x\in (X_k)_{sing}}H^1(x)}\to Pic(X_k)/2\to Pic(\widetilde{X_k^{red}})/2\to0\]
by Proposition \ref{picard}. Recall that the \(G(X_k)\) now by definition is the subgroup of \(H^1(k,{_2}G)=Coker(j)\) generated by \(x\in p^{-1}((X_k)_{sing})\) with \(\sqrt{-1}\notin k(x)\). Also recall the computation of \(S(X_k)\) from Proposition \ref{unramified}:
\[H^0(\widetilde{X^{red}})\to\frac{\oplus_{x'\in p^{-1}((X_k)_{sing})}H^0(x')}{\oplus_{x\in (X_k)_{sing}}H^0(x)}\to S(X_k)^{\vee}\to0.\]

Finally we have the following tables:
\begin{center}
\begin{tabular}{|c|c|c|c|c|c|c|c|c|c|c|}
\hline
Type						&\(I_0\)&\(I_{1,2}\)	&\(II\)	&\(I_n,n\geq1\)	&\(I_{2n-2,2},I_{2n-1,2}\)	&\(IV\)	&\(IV_2\)	&\(III\)\\
\hline
\(dim(S(X_k))\)		&0			&0					&0			&1						&0										&0			&0				&0\\
\hline
\(dim(G(X_k))\)		&0			&0					&0			&1						&0										&0			&0				&0\\
\hline
\(dim(Pic(X_k)/2)\)&1,2,3	&2					&1			&n+1					&n+1									&3			&2				&2\\
\hline
\(tr(X_k)\)			&1			&1					&1			&n 					&2										&3			&2				&2\\
\hline
\(a\)						&2			&2					&2			&n+2					&3										&4			&3				&3\\
\hline
\(b\)						&0,2,4	&2					&0			&1						&2n-2									&0			&0				&0\\
\hline
\end{tabular}
\begin{tabular}{|c|c|c|c|c|c|c|c|c|c|}
\hline
Type						&\(I_{0,2}^*\)	&\(I_{0,3}^*\)	&\(I_{n-5}^*,n\geq5\)	&\(I_{n-4,2}^*\)&\(IV^*\)	&\(IV_2^*\)	&\(III^*\)	&\(II^*\)\\
\hline
\(dim(S(X_k))\)		&0						&0						&0								&0						&0				&0					&0				&0\\
\hline
\(dim(G(X_k))\)		&0						&0						&0								&0						&0				&0					&0				&0\\
\hline
\(dim(Pic(X_k)/2)\)&4						&3						&n 							&n						&7				&5					&8				&9\\
\hline
\(tr(X_k)\)			&3						&3						&n 							&n-1					&7				&3					&8				&9\\
\hline
\(a\)						&4						&4						&n+1							&n 					&8				&4					&9				&10\\
\hline
\(b\)						&2						&0						&0								&2						&0				&4					&0				&0\\
\hline
\end{tabular}
\end{center}
\section{Theta Characteristics of Hyperelliptic Reductions}\label{hyperelliptic}
If \(X_K\) is hyperelliptic with genus \(\geq2\) the degree \(2\) covering \(r:X_K\to\mathbb{P}^1_K\) satisfies \(\omega_{X_K}=r^*O(g(X_K)-1)\) hence after composing with a \(\mathbb{P}^1\to\mathbb{P}^1\) of degree \(g(X_K)-1\) we have \(r^*O(1)=\omega_{X_K}\).

If \(X_K\) is not hyperelliptic, we may suppose that there is a linear subspace of dimension \(g(X_K)-3\) being disjoint from the image of the canonical embedding \(X_K\to\mathbb{P}^{g(X_K)-1}\) so we get a map \(X_K\to\mathbb{P}^1\) pulling back \(O(1)\) to \(\omega_{X_K}\).

If these maps can be lifted to a regular model, after reduction we get a map \(q:X_k\to\mathbb{P}^1_k\) such that \(q^*O(1)=\omega_{X_k}\).

Suppose that we have maps \(\widetilde{X_k^{red}}\xrightarrow{p}X_k\xrightarrow{q}\mathbb{P}^1\) where \(q\) is finite flat, \(q^*O(1)=\omega_{X_k}\) and that the \(\widetilde{X_k^{red}}\) is a disjoint union of either (geometrically connected) hyperelliptic curves \(X_a\) written as
\[Spec(k[x,y]/(y^2=\epsilon P_1(x)\cdots P_s(x)))\]
where \(\epsilon\in k^{\times}\) and \(\{P_i\}\) are distinct monic irreducible polynomials with degree \(d_i\), or simply \(\mathbb{P}^1_k\). The \(d=\sum_{i=1}^sd_i\) is \(2g+1\) or \(2g+2\) where \(g=g(X_a)\) is the genus. Denote by \(r\) unique covering \((x,y)\mapsto(x:1)\) over \(\mathbb{P}^1\) of degree \(2\) and by \(e=[K(X_a):K(\mathbb{P}^1)]\). Furthermore we define
\[\delta=\begin{cases}\sqrt{\epsilon}&\sqrt{\epsilon}\in k\\\sqrt{-\epsilon}&\sqrt{-1}\in k,\sqrt{\epsilon}\notin k\\a+b\sqrt{-1}&a^2+b^2=\epsilon;a,b\in k;\sqrt{\epsilon},\sqrt{-1}\notin k\end{cases},\]
which is the solution of \(\delta\sigma(\delta)=\epsilon\) on the quadratic extension of \(k\).

\begin{lemma}\label{factor}
\footnote{\url{https://mathoverflow.net/questions/495689/maps-from-a-hyperelliptic-curve-to-mathbbp1}}
If \(e\leq g\) and \(e\) is even, the \((q\circ p)|_{X_a}\) factors through \(r\).
\end{lemma}
\begin{proof}
Set \(f=(q\circ p)|_{X_a}\). Since \(h^0(X_a,f^*O(1))\geq 2\) and \(h^1(X_a,f^*O(1))>1+g-1-e=g-e\geq0\) by Riemann-Roch, the \(r_*f^*O(1)\) is of the form \(O(a)\oplus O(b)\) where \(a\geq 1,b\leq-2\). Hence
\[H^0(X_a,r^*O(b))=H^0(\mathbb{P}^1,r_*r^*O(b))=H^0(\mathbb{P}^1,O(b)\oplus O(b-g-1))=0,\]
\[H^0(\mathbb{P}^1,O(a))=H^0(X_a,f^*O(1))=H^0(\mathbb{P}^1,r_*f^*O(1))\xrightarrow[\cong]{r^*}H^0(X_a,r^*r_*f^*O(1)).\]
The map \(H^0(\mathbb{P}^1,O(1))\xrightarrow{f^*}H^0(X_a,f^*O(1))\) then gives two base point free sections \(u,v\) of \(f^*O(1)\). The last equality shows that they are equal to \(r^*s,r^*t\) respectively, where \(s,t\in H^0(\mathbb{P}^1,r_*f^*O(1))=H^0(\mathbb{P}^1,O(a))\). By faithfully flatness of \(r\) we see that \(s,t\) are base point free over \(\mathbb{P}^1\). Hence we obtain a map \(h:\mathbb{P}^1\to\mathbb{P}^1\) given by \(s,t\), which satisfy \(h\circ r=f\).
\end{proof}
Let us recall the results in \cite[\S2]{C}. Let \(t_{d_1+\cdots+d_{i-1}+1},\cdots,t_{d_1+\cdots+d_{i}}\) be roots of \(P_i\) over \(\bar{k}\), such that the Frobenius action is the permutation
\[(t_{d_1+\cdots+d_{i-1}+1}\cdots t_{d_1+\cdots+d_{i}}).\]
If \(d=2g+1\), the \(_2Pic(\widetilde{X_{\bar{k}}})\) has a basis \(D_i=t_i-\infty,i=1,\cdots,d-1\). Otherwise it has a basis \(D_i=t_i-t_1,i=2,\cdots,d-1\).
\begin{proposition}\label{canonical}
Suppose that \(\lambda\in\mathbb{P}^1_k\) with \(deg(\lambda)=n\) being odd (see Definition \ref{poly} for \(P_{\lambda}\)). For each component \(X_a\) in the context above, we have
\begin{enumerate}
\item If \(d\) is odd, we have
\[r^*\lambda=div(P_{\lambda}(x))+2n\infty.\]
\item If \(d\) is even and \(d_1\) is odd, we have
\[r^*\lambda=div(\frac{P_{\lambda}(x)}{P_1(x)})-2\cdot\frac{n-1}{2}div(P_1(x))+2n(t_1+\cdots+t_{d_1})-2\cdot\frac{d_1-1}{2}(r^*\lambda-div(P_{\lambda}(x))).\]
\item Denote by \(Q(x)=\delta\prod_{i=1}^{g+1}(x-t_{2i-1})\) and by \(d\in k\) some non-square element. Write
\[Q(x)=A(x)+\sqrt{d}B(x)\]
for \(A(x),B(x)\in k[x]\), where the \(B(x)\) has zeros \(\{w_j\}\), being solutions of \(\sigma(Q)(x)=Q(x)\). Define
\[\varphi=2y+Q(x)+\sigma(Q)(x).\]
We have
\[div(\varphi)=\begin{cases}2\sum_j(w_j,-Q(w_j))-(g+1)\infty_{+}+(g+1-2deg(B(x)))\infty_{-}&\sqrt{\epsilon}\in k\\2\sum_j(w_j,-Q(w_j))-(g+1)(\infty_{+}+\infty_{-})&\sqrt{\epsilon}\notin k\end{cases}.\]

If \(d_1,\cdots,d_s,g\) are even, we have
\[r^*\lambda=div(\frac{P_{\lambda}(x)}{\varphi})+2D\]
where
\[D=\sum_{i=1}^{g+1}t_{2i-1}+\frac{n-g-1}{2}(\infty_{+}+\infty_{-})-div(\frac{Q(x)}{Q(x)+y})\]
is defined over \(k\).
\item If \(d_1,\cdots,d_s\) are even and \(g\) is odd, the \(r^*\lambda\) is not a square in \(Pic(X_a)\).
\end{enumerate}
\end{proposition}
\begin{proof}
\begin{enumerate}
\item We have
\[r^*\lambda=r^*(div(P_{\lambda}(x))+n\infty)=div(P_{\lambda}(x))+2n\infty.\]
\item We have
\[d_1r^*\lambda=r^*(div(\frac{P_{\lambda}(x)^{d_1}}{P_1(x)^{n}})+div(P_1(x)^n)+nd_1\infty)=div(\frac{P_{\lambda}(x)^{d_1}}{P_1(x)^{n}})+2n(t_1+\cdots+t_{d_1}).\]
The statement follows easily.
\item We have \(div(y)=\sum_{i=1}^{2g+2}t_i-(g+1)(\infty_{+}+\infty_{-})\), which gives
\begin{align*}
	&\sigma(D_{2g+1})\\
=	&t_{2g+2}-t_2\\
=	&div(y)-t_1-2t_2-\sum_{i=3}^{2g+1}t_i+(g+1)(\infty_{+}+\infty_{-})\\
=	&div(y)-2gt_1-2t_2-\sum_{i=3}^{2g+1}D_i+(g+1)(\infty_{+}+\infty_{-})\\
=	&\sum_{i=3}^{2g+1}D_i-2\sum_{i=3}^{2g+1}D_i+div(y)-2gt_1-2t_2+(g+1)(\infty_{+}+\infty_{-})\\
=	&\sum_{i=3}^{2g+1}D_i-div(\frac{\prod_{i=3}^{2g+1}(x-t_i)}{(x-t_1)^{2g-1}})+div(y)-div((x-t_1)^g(x-t_2))\\
=	&\sum_{i=3}^{2g+1}D_i+div(\frac{y(x-t_1)^{g-1}}{\prod_{i=3}^{2g+1}(x-t_i)}).
\end{align*}
We verify that the \(\sum_{i=2}^{g+1}D_{2i-1}+nt_1\) lives in \(Pic(X_a\times_k\bar{k})^{\sigma}\), whose square is linearly equivalent to \(r^*\lambda\) over \(\bar{k}\). We have
\begin{align*}
	&(\sigma-Id)(\sum_{i=2}^{g+1}D_{2i-1}+nt_1)\\
=	&\sum_{i=2}^gD_{2i}-D_{2i-1}+\sigma(D_{2g+1})-D_{2g+1}+(n-g+1)D_2\\
=	&\sum_{i=2}^gD_{2i}-D_{2i-1}+\sum_{i=3}^{2g}D_i+div(\frac{y(x-t_1)^{g-1}}{\prod_{i=3}^{2g+1}(x-t_i)})+(n-g+1)D_2\\
=	&2\sum_{i=2}^gD_{2i}+(n-g+1)D_2+div(\frac{y(x-t_1)^{g-1}}{\prod_{i=3}^{2g+1}(x-t_i)})\\
=	&div(\frac{\prod_{i=2}^g(x-t_{2i})}{(x-t_1)^{g-1}})+div(\frac{x-t_2}{x-t_1})^{\frac{n-g+1}{2}}+div(\frac{y(x-t_1)^{g-1}}{\prod_{i=3}^{2g+1}(x-t_i)})\\
=	&div(\left(\frac{x-t_2}{x-t_1}\right)^{\frac{n-g-1}{2}}\frac{y}{\delta\prod_{i=1}^{g+1}(x-t_{2i-1})})\\
=	&(\sigma-Id)div((x-t_1)^{\frac{n-g-1}{2}}\frac{\delta\prod_{i=1}^{g+1}(x-t_{2i-1})}{\delta\prod_{i=1}^{g+1}(x-t_{2i-1})+y}),
\end{align*}
which shows that
\[D=\sum_{i=1}^{g+1}t_{2i-1}-div(\frac{\delta\prod_{i=1}^{g+1}(x-t_{2i-1})}{\delta\prod_{i=1}^{g+1}(x-t_{2i-1})+y})+\frac{n-g-1}{2}(\infty_{+}+\infty_{-})\]
is a square root of \(r^*\lambda\) and lives in \(Div(X_a)\). Set \(f=(x-t_1)^{\frac{n-g-1}{2}}\frac{\delta\prod_{i=1}^{g+1}(x-t_{2i-1})}{\delta\prod_{i=1}^{g+1}(x-t_{2i-1})+y}\). Finally we have
\begin{align*}
	&r^*\lambda\\	
=	&div(P_{\lambda}(x))+n(\infty_{+}+\infty_{-}))\\
=	&div(\frac{P_{\lambda}(x)}{(x-t_1)^n})+2(nt_1+\sum_{i=2}^{g+1}D_{2i-1}-div(f))-2\sum_{i=2}^{g+1}D_{2i-1}+2div(f)\\
=	&div(\frac{P_{\lambda}(x)}{(x-t_1)^n})+2D-2\sum_{i=2}^{g+1}D_{2i-1}+2div(f)\\
=	&div(\frac{P_{\lambda}(x)(x-t_1)^{g-n}}{\sum_{i=2}^{g+1}(x-t_{2i-1})})+2D+2div(f)\\
=	&div(\frac{P_{\lambda}(x)}{\delta\prod_{i=1}^{g+1}(x-t_{2i-1})+\sigma(\delta)\prod_{i=1}^{g+1}(x-t_{2i})+2y})+2D.
\end{align*}

Denote by \(Q(x)=\delta\prod_{i=1}^{g+1}(x-t_{2i-1})\). Let us compute \(div(Q(x)+\sigma(Q)(x)+2y)\). Suppose
\[\begin{array}{cc}Q(x)+\sigma(Q)(x)+2y=0&y^2=Q(x)\sigma(Q)(x)\end{array}.\]
We obtain
\begin{equation}\label{Q}(Q(x)-\sigma(Q)(x))^2=0.\end{equation}
Write \(Q(x)=A(x)+\sqrt{d}B(x)\) as in the statement. The \eqref{Q} becomes
\[4dB(x)^2=0.\]
So the \(\varphi=Q(x)+\sigma(Q)(x)+2y\) has order \(2\) at each root of \(B(x)\) and \(y=-Q(x)\) at these roots. Here \(deg(B(x))=g+1\) if \(\sqrt{\epsilon}\notin k\). At infinity we have
\[ord_{\infty_{+}}(\varphi)=-g-1\]
\[ord_{\infty_{-}}(\varphi)=\begin{cases}g+1-2deg(B(x))&\sqrt{\epsilon}\in k\\-g-1&\sqrt{\epsilon}\notin k\end{cases}.\]
Hence
\[div(\varphi)=\begin{cases}2\sum_j(w_j,-Q(w_j))-(g+1)\infty_{+}+(g+1-2deg(B(x)))\infty_{-}&\sqrt{\epsilon}\in k\\2\sum_j(w_j,-Q(w_j))-(g+1)(\infty_{+}+\infty_{-})&\sqrt{\epsilon}\notin k\end{cases}.\]
This concludes the proof.
\item We want to show that the equation
\[(\sigma-Id)\sum_{i=2}^{2g+1}x_iD_i\sim D_2,x_i\in\mathbb{Z}/2\]
has no solution so the square root of \(r^*\lambda\) is not defined over \(k\). We use the explicit Frobenius matrix (more precisely, \(\sigma-Id\)) given in \cite[Lemma 2.4]{C}.

Suppose \(s>1\). The \(P_s(x)\) gives the homogeneous linear equations represented by the matrix
\[\begin{pmatrix}1&&&&1\\1&1&&&1\\&1&1&&1\\&&\vdots&&\\&&1&1&1\\&&&1&0\end{pmatrix}_{(d_s-1)\times(d_s-1)},\]
whose solution is given by \(\footnotesize\begin{pmatrix}\alpha&0&\alpha&\cdots&0&\alpha\end{pmatrix}^T\).

For \(P_j(x)\) where \(2\leq j<s\), we have the homogeneous linear equations represented by the matrix
\[\begin{pmatrix}1&&&&1&0&\cdots&0&1\\1&1&&&&0&\cdots&0&1\\&1&1&&&0&\cdots&0&1\\&&\vdots&&&&\vdots&&\vdots\\&&1&1&&&\vdots&&1\\&&&1&1&0&\cdots&0&1\end{pmatrix}_{d_j\times(d_j+\cdots+d_s-1)},\]
whose solution of the first \(d_j\) variables is given by \(\footnotesize\begin{pmatrix}\beta+\alpha&\beta&\cdots&\beta+\alpha&\beta\end{pmatrix}^T\).

For \(P_1(x)\) we have the inhomogeneous linear equations represented by the matrix
\[\begin{pmatrix}0&1&1&\cdots&1&0&1\\1&1&0&\cdots&0&1&0\\&1&1&\cdots&0&1&0\\&&&\vdots&&\end{pmatrix}_{(d_1-1)\times(d-1)},\]
where the first row gives the equation
\[(\frac{d_1}{2}-1+\frac{d_s}{2}-1+\sum_{j=2}^{s-1}\frac{d_j}{2})\alpha=(g-1)\alpha=1.\]
So if \(g\) is odd there is no solution.

If \(s=1\), the \(P_1(x)\) gives the inhomogeneous linear equations represented by the matrix
\[\begin{pmatrix}0&1&1&\cdots&1&1&0&1\\1&1&0&\cdots&0&0&1&0\\&1&1&\cdots&0&0&1&0\\&&&\vdots&&&\\&&&&1&1&1&0\\&&&&&1&0&0\end{pmatrix}_{(d_1-2)\times(d_1-1)},\]
whose solution is given by \(\footnotesize\begin{pmatrix}0&\alpha&\cdots&0&\alpha\end{pmatrix}^T\). The first row gives the equation
\[(\frac{d-2}{2}-1)\alpha=(g-1)\alpha=1,\]
which has no solution if \(g\) is odd.
\end{enumerate}
\end{proof}
Combining Proposition \ref{square}, \ref{sq}, \ref{canonical} and Remark \ref{factor1} we have:
\begin{coro}\label{algorithm}
In the context above, denote by \(p_a=(q\circ p)|_{X_a}:X_a\to\mathbb{P}^1\). If one of \([K(X_a):K(\mathbb{P}^1)]\) is odd, \(\Theta(\omega_{X_K})=\emptyset\). Otherwise suppose that \([K(X_a):K(\mathbb{P}^1)]\leq g(X_a)\) for every \(a\).
\begin{enumerate}
\item If \(X_a\) is ramified at \(\infty\), set \(f_a=1\).
\item If \(X_a\) is unramified at \(\infty\) and there is an odd degree ramification point with prime polynomial \(P_1\) (namely \(deg(P_1)\) is odd), set \(f_a=P_1\).
\item\label{easy} If \(X_a\) has no odd degree ramifications, set
\[f_a=2y+\delta\prod_{i=1}^{g+1}(x-t_{2i-1})+\sigma(\delta)\prod_{i=1}^{g+1}(x-t_{2i}).\]
if \(g(X_a)\) is even, otherwise \(\Theta(\omega_{X_K})=\emptyset\).
\item If \(Y_b\subseteq\widetilde{X_k^{red}}\) is a component being isomorphic to \(\mathbb{P}^1_k\), it gives a rational function \(\frac{U(x)}{V(x)}\) where \(U,V\in k[x]\) are coprime. Set \(f_b=V\) if every \([K(Y_b):K(\mathbb{P}^1)]\) is even, otherwise \(\Theta(\omega_{X_K})=\emptyset\).
\end{enumerate}
Denote by \(f=(\frac{P_{\lambda}\circ p_a}{f_a},\frac{P_{\lambda}\circ(q\circ p)|_{Y_b}}{f_b})\in K(\widetilde{X_k^{red}})\). Take \(\lambda\in\mathbb{P}^1\) with sufficiently large odd degree so that \(Supp(q^*\lambda)\cap X_{sing}=\emptyset\). Then the class
\[\Lambda(\omega_{X_k})=\overline{\oplus_{p(x)\in (X_k)_{sing}}f(x)}\in H^1(k,{_2}G)/R,\]
is zero if and only \(\Theta(\omega_{X_K})\neq\emptyset\).
\end{coro}
For the computation of \(R\) in our case, one may use the explicit basis of \(_2Pic(X_a)\) given in \cite[Lemma 2.6]{C}.

{}
\end{document}